\documentclass[11pt, a4paper, reqno]{amsart}

\usepackage[margin=2.5cm]{geometry}

\usepackage[utf8]{inputenc}
\usepackage[T1]{fontenc}
\usepackage{color}
\usepackage[foot]{amsaddr}
\usepackage{float}

\usepackage[colorlinks, linkcolor=blue]{hyperref}
\usepackage{xurl} 
\hypersetup{breaklinks=true} 

\usepackage[url=false, style=numeric, backend=biber,firstinits=true, maxbibnames=99, maxalphanames=4]{biblatex}
\DeclareFieldFormat{issn}{}
\usepackage[capitalize]{cleveref}
\crefname{equation}{}{}

\Crefname{assumption}{Assumption}{Assumptions}

\newcommand{\nocontentsline}[3]{}
\newcommand{\tocless}[2]{\bgroup\let\addcontentsline=\nocontentsline#1{#2}\egroup}

\makeatletter
\def\@tocline#1#2#3#4#5#6#7{\relax
  \ifnum #1>\c@tocdepth 
  \else
    \par \addpenalty\@secpenalty\addvspace{#2}%
    \begingroup \hyphenpenalty\@M
    \@ifempty{#4}{%
      \@tempdima\csname r@tocindent\number#1\endcsname\relax
    }{%
      \@tempdima#4\relax
    }%
    \parindent\z@ \leftskip#3\relax \advance\leftskip\@tempdima\relax
    \rightskip\@pnumwidth plus4em \parfillskip-\@pnumwidth
    #5\leavevmode\hskip-\@tempdima
      \ifcase #1
       \or\or \hskip 1em \or \hskip 2em \else \hskip 3em \fi%
      #6\nobreak\relax
    \hfill\hbox to\@pnumwidth{\@tocpagenum{#7}}\par
    \nobreak
    \endgroup
  \fi}
\makeatother

\newtheorem{theorem}{Theorem}

\newenvironment{delayedproof}[1]
 {\begin{proof}[\raisedtarget{#1}Proof of \Cref{#1}]}
 {\end{proof}}
\newcommand{\raisedtarget}[1]{%
  \raisebox{\fontcharht\font`P}[0pt][0pt]{\hypertarget{#1}{}}%
}

\newtheorem{proposition}{Proposition}[section]
\newtheorem{lemma}[proposition]{Lemma}
\newtheorem{corollary}[proposition]{Corollary}
\theoremstyle{definition}
\newtheorem{definition}[proposition]{Definition}
\newtheorem{remark}[proposition]{Remark}

\newtheorem{assumption}{Assumption}
\numberwithin{equation}{section}

\usepackage{amssymb, mathrsfs}
\usepackage{bm}
\def \R {\mathbb{R}}
\def \rmH {\mathrm{H}}
\def \rmd {\mathrm{d}}
\def \L {\mathscr{L}}
\def \K {\mathscr{K}}
\def \rmL {\mathrm{L}}
\def \calA {\mathcal{A}}
\def \calB {\mathcal{B}}
\def \calD {\mathcal{D}}

\def \calA {\mathcal{A}}

\def \rmV {\mathrm{V}}
\def \rmW {\mathrm{W}}
\def \capa {\mathrm{cap}}

\usepackage{mathtools}
\DeclarePairedDelimiter\abs{\lvert}{\rvert}
\DeclarePairedDelimiter\norm{\lVert}{\rVert}

\DeclarePairedDelimiter{\pair}{\langle}{\rangle}

\makeatletter

\newcommand{\lefteqno}{\let\veqno\@@leqno}
\renewcommand{\eqref}[1]{\textup{(\ignorespaces\ref{#1}\unskip\@@italiccorr)}}

\newcommand*{\transpose}{%
  {\mathpalette\@transpose{}}%
}
\newcommand*{\@transpose}[2]{%
  \raisebox{\depth}{$\m@th#1\intercal$}%
}

\makeatother

\begin{document}


\title[Variational Capacity for Non-Self-adjoint Fokker-Planck Operators]{Variational Formulation and Capacity Estimates for Non-Self-Adjoint Fokker-Planck Operators in Divergence Form}

\author{Mingyi Hou}
\address{Mingyi Hou\\Department of Mathematics, Uppsala University\\
751 05 Uppsala, Sweden}
\email{mingyi.hou@math.uu.se}

\date{\today}

\begin{abstract}
    We introduce a variational formulation for a general class of possibly degenerate, non-self-adjoint Fokker-Planck operators in divergence form, motivated by the work of Albritton et al.~(2024), and prove that it is suitable for defining the variational capacity. 
    Using this framework, we establish rough estimates for the equilibrium potential in the elliptic case, providing a novel approach compared to previous methods. 
    Finally, we derive the Eyring-Kramers formula for non-self-adjoint elliptic Fokker-Planck operators in divergence form, extending the results of Landim et al.~(2019) and Lee \& Seo (2022).
\end{abstract}

\maketitle

\tableofcontents

\section{Introduction and Main Results}

In this paper, we introduce a variational capacity formulation for a general class of Fokker-Planck-Kolmogorov equations in divergence form. Additionally, we prove the Eyring-Kramers formula in elliptic, non-reversible cases. 

The operator we study is given by, for $z \in \mathbb{R}^d$,
\begin{equation}\label{eq:main}
\L_\varepsilon f(z) = \varepsilon e^{W/\varepsilon}\nabla\cdot\left[e^{-W/\varepsilon}\mathbf{A}\nabla f\right] + \left(\bm{b} + \mathbf{B}z\right)\cdot\nabla f,
\end{equation}
where $W(z):\mathbb{R}^d\to \mathbb{R}$ is a confining potential that is differentiable, and $\varepsilon$ is a parameter corresponding to thermal noise. Moreover, $\mathbf{A}(z)$ is a matrix, $\bm{b}(z): \mathbb{R}^d\to\mathbb{R}^d$ is a vector field, and $\mathbf{B}$ is a constant matrix of the following form:
\begin{equation}\label{assump:coeff}
    \lefteqno
    \tag{H1}
    \mathbf{A}(z) = \begin{bmatrix}
        \mathbf{0} & \mathbf{0} 
        \\
        \mathbf{0} & \widehat{\mathbf{A}}(z)
    \end{bmatrix},
    \quad
    \bm{b}(z) = \begin{bmatrix}
        \mathbf{0} \\
        \widehat{\bm{b}}(z)
    \end{bmatrix},
    \quad 
    \mathbf{B} = \begin{bmatrix}
        \mathbf{0} & \mathbf{B}_\kappa &  &  & 
        \\
        & \mathbf{0} & \mathbf{B}_2 &  &  
        \\
        &  & \ddots & \ddots & 
        \\
        & & &\ddots & \mathbf{B}_{1}
        \\
        &  &  &  & \mathbf{0}
    \end{bmatrix}
\end{equation}
where 
\begin{itemize}
    \item $\widehat{\mathbf{A}}(z)$ is a bounded $n_0\times n_0$ matrix for fixed $n_0\leq d$, and is H\"older continuous of order $\alpha$ for some $0<\alpha<1$, and uniformly elliptic, i.e.~$\xi^\transpose\widehat{\mathbf{A}}\xi \geq \lambda \abs{\xi}^2$ for some $\lambda>0$;
    \item $\widehat{\bm{b}}(z)$ is a continuous vector field of length $n_0$ and is locally bounded;
    \item Each $\mathbf{B}_j$ is an $n_{j}\times n_{j-1}$ constant matrix of rank $n_{j}$ such that $n_0 + n_1 + \cdots + n_\kappa = d$.
\end{itemize}

Additionally, we denote $\widehat{\nabla} := (\partial_{d - n_0 + 1},\dots,\partial_{d} )$, the gradient of the last $n_0$ coordinates.

The following weighted divergence-free assumption is imposed on $\bm{b}(z)$ and $\mathbf{B}$:
\begin{equation}\label{assump:divfree}
    \lefteqno
    \tag{H2}
    \nabla\cdot \left( \left(\bm{b}(z) + \mathbf{B}z\right)W(z)\right) = 0, \quad \text{for all } z\in\mathbb{R}^d.
\end{equation}
As a consequence, there exists a stationary measure $\rho_\varepsilon$ for the operator $\L_\varepsilon$ given by
\begin{equation*}
    \rho_\varepsilon(z) = \frac{1}{Z_\varepsilon} e^{-W(z)/\varepsilon}\rmd z, \quad \text{where } Z_\varepsilon = \int_{\mathbb{R}^d} e^{-W(z)/\varepsilon} \rmd z,
\end{equation*}
which satisfies $\int_{\mathbb{R}^d} \L_\varepsilon f\, \rho_\varepsilon(\rmd z) = 0$ for all $f$ such that the integral is well-defined. 
We denote by $\L_\varepsilon^\dag$ the adjoint operator of $\L_\varepsilon$ with respect to the measure $\rho_\varepsilon$, given by
\begin{equation}\label{eq:adjoint}
    \L_\varepsilon^\dag f(z) = \varepsilon e^{W/\varepsilon}\nabla\cdot\left[e^{-W/\varepsilon}\mathbf{A}^\transpose\nabla f\right] - \left(\bm{b} + \mathbf{B}z\right)\cdot\nabla f.
\end{equation}

\tocless{\subsection*{Motivation and Brief History}}
Roughly speaking, if $\mathbf{A}$ is divergence-free, denote by $\mathbf{S} = (\mathbf{A} + \mathbf{A}^\transpose)/2$ the symmetric part. Then, the operator \cref{eq:main} is the generator of the following Markov process:
\begin{equation}\label{eq:process}
    \rmd Z_t = \left( -\mathbf{A}^\transpose\nabla W(Z_t) + \bm{b}(Z_t) + \mathbf{B} Z_t \right) \rmd t + \sqrt{2\varepsilon\mathbf{S}} \rmd B_t,
\end{equation}
where $B_t$ is the standard Brownian motion in $\mathbb{R}^d$. 
This process is also known as the overdamped or underdamped Langevin process, depending on whether $\mathbf{A}$ has full rank.
On one hand, if $\mathbf{A} = \mathbf{I}_d$ is the identity matrix, $\mathbf{B} = 0$, and $\bm{b} = 0$, then \cref{eq:process} becomes the reversible overdamped Langevin process:
\begin{equation*}
    \rmd Z_t = -\nabla W(Z_t) \rmd t + \sqrt{\varepsilon} \rmd B_t,
\end{equation*} 
associated with the Fokker-Planck operator
\begin{equation*}
    \L_\varepsilon f= \varepsilon \Delta f - \nabla W\cdot\nabla f.
\end{equation*}
On the other hand, if $d = 2n$ for some positive integer $n$, $z = (x, v) \in \mathbb{R}^n\times\mathbb{R}^n$, assume now that $W(z) = U(x) + \abs{v}^2/2$, where $U(x):\mathbb{R}^n\to \mathbb{R}$ is a confining potential. Let
\begin{equation*}
    \mathbf{A} = \begin{pmatrix}
        0 & 0 \\
        0 & \mathbf{I}_n
    \end{pmatrix},
    \quad 
    \bm{b} = \begin{pmatrix}
        0 \\
        - \nabla_x U
    \end{pmatrix},
    \quad 
    \mathbf{B}
    =\begin{pmatrix}
        0 & \mathbf{I}_n \\
        0 & 0
    \end{pmatrix},
\end{equation*}
then \cref{eq:process} becomes underdamped Langevin process:
\begin{equation*}
    \begin{cases}
        \rmd X_t = V_t \rmd t,\\
        \rmd V_t = - \left(\nabla_x U(X_t) + V_t\right)\rmd t + \sqrt{\varepsilon} \rmd B_t.
    \end{cases}
\end{equation*}
associated with the Kramers-Fokker-Planck (KFP) operator
\begin{equation*}
    \K_\varepsilon f = \varepsilon \Delta_v f - v\cdot\nabla_v f - \nabla_x U\cdot\nabla_v f + v\cdot\nabla_x f.
\end{equation*}

The processes described above are good continuous models for stochastic gradient descent, with or without momentum. Metastability is crucial for the training process; see Shi et al.~\cite{SSJ23} and references therein.

The Eyring-Kramers formula refines the Arrhenius law and provides better asymptotics for the mean transition time between two wells. The sharp asymptotics for the reversible case were proved in \cite{BEGK04}, based on a potential-theoretic approach and variational capacity. In the subsequent work \cite{BGK05}, Bovier et al.~showed the relation between the principal eigenvalue and the capacity, as well as sharp asymptotics for low-lying eigenvalues and eigenfunctions. Since then, this has been a rich research area with strong interest in the literature.
Notably, for the case where $\mathbf{A}$ is elliptic non-symmetric and $\bm{b} = 0$, $\mathbf{B} = 0$, the Eyring-Kramers formula was proved in \cite{LMS19} using a variational formula and potential theory approach. For the case where $\mathbf{A} = \mathbf{I}_d$, $\bm{b} \neq 0$, the Eyring-Kramers formula was proved in \cite{LS22} using a non-variational approach, though the key observation is similar to \cite{LMS19}. Additionally, for the reversible case where $W$ has degenerate or singular critical points, the Eyring-Kramers formula was studied in \cite{BG10,AJV23,AJ24} and also in Kramers' original work \cite{Kra40}.

On the other hand, the Eyring-Kramers formula for the low-lying spectrum of \cref{eq:main} was recently proved in \cite{BLPM22} using semiclassical analysis. Semiclassical analysis and other methods for proving the Eyring-Kramers formula can be found in Berglund's survey \cite{Ber11} and references therein. 
A potential theory approach to the mean transition time for \cref{eq:main} remains open, while we note there are researches going on in the field, e.g.~\cite{LAS23}. The potential theory framework and variational formulation for \cref{eq:main} are far from complete.

Meanwhile, if $\mathbf{A}$ is only assumed to be H\"older continuous, the operator \cref{eq:main} can no longer be viewed as the generator of an It\^o process, making the mean transition time less meaningful. However, we can still argue that the Eyring-Kramers formula remains valid. Indeed, the mean exit time $w(z) = \mathbb{E}^z[\tau_{\calB}]$, where $\tau_{\calB}$ is the stopping time for the process \cref{eq:process} to hit $\calB\subset\mathbb{R}^d$, satisfies the following problem:
\begin{equation}\label{eq:landscape}
    \begin{cases}
        -\L_\varepsilon w = 1, &\textrm{ in } \overline{\calB}^\complement,\\
        \hfill w = 0, & \textrm{ on } \calB.
    \end{cases}
\end{equation}
Interestingly, $w$ is also known as the landscape function and is a popular topic when $W$ is a random potential; see, e.g.~\cite{FM12,ADFJ19}.
The relation between the Eyring-Kramers formula and the landscape function is that $w(\bm{m}_1) = \mathbb{E}^{\bm{m}_1}[\tau_{B_\varepsilon(\bm{m}_0)}]$, where $\bm{m}_0,\bm{m}_1$ are two local minima of $W$.
By the Donsker-Varadhan inequality, e.g.~\cite{LS17}, the principal eigenvalue $\lambda_2$ (if it exists) is bounded below by:
\begin{equation*}
    \lambda_2 \geq \frac{1}{\sup_{z\in\overline{\calB}^\complement}w(z)}.
\end{equation*}
In the reversible case, Bovier et al.~\cite{BGK05} improved this relation and obtained $\lambda_2 \approx  \capa / \|h\|_2^2$, where $\capa$ is the capacity between two wells and $h$ is the equilibrium potential. A similar relation for the non-reversible case remains an open problem. Further studies on the relation between mean exit time and the principal eigenvalue in the non-reversible case can be found in, for example, \cite{LPMN24}.

\subsection{Contributions and Novelty}

In this paper, we motivate the definition of capacity for \cref{eq:main} and propose a novel variational capacity formula in the form of an obstacle problem. The formula generalizes the equivalent one proposed in \cite{LMS19} for the non-degenerate case, while being inspired by the recent study of variational methods for the Dirichlet problem of \cref{eq:main} in \cite{AAMN21}.

We use the variational capacity formula to establish rough a priori estimates on the equilibrium potential for the non-degenerate case. Although these rough estimates are known, e.g.~\cite{LMS19,LS22}, the application of variational capacity is novel and may be extendable to degenerate cases.

Lastly, by combining a perturbation technique, we apply the rough estimates and follow the constructive method from \cite{LS22} to derive sharp capacity estimates for elliptic non-reversible divergence-form equations where $\mathbf{A}$ is only H\"older continuous.

\subsection{Preliminaries}
The confining condition for $W(z)$ is given by 
\begin{equation}\label{assump:confining}
    \lefteqno
    \tag{HC}
    \lim_{\abs{z}\to\infty} W(z) = \infty, \textrm{ and } W(z)\geq c_1 \abs{x}^{q} -c_2,
\end{equation}
for some $q>1$ and some positive constants $c_1,\,c_2$.
By hypocoercivity theory \cite{Vil09,BDZ24}, this condition ensures that the generated semigroup (or Markov process) converges to equilibrium exponentially (or that the process is recurrent).

{
\subsubsection*{Function Space Basics}
Let $\calD\subset\mathbb{R}^d$ be an open set. 
We denote by $\rmH^1(\calD) = \rmW^{1,2}(\calD)$ the Sobolev space of functions with respect to the measure $\rho_\varepsilon$, and $\rmH^1_0(\calD)$ be the closure of $\mathrm{C}_c^\infty(\calD)$ in $\rmH^1(\calD)$.
We also introduce the function space $\widehat{\rmH}^1(\calD):= \{f(z) \in \rmL^2(\calD,\rho_\varepsilon): \widehat{\nabla} f\in\rmL^2(\calD,\rho_\varepsilon) \}$, equipped with the weighted norm, and $\widehat{\rmH}^1_0$ as the closure of $\mathrm{C}_c^\infty(\calD)$. 
In addition, we define $\widehat{\rmH}^1_{\mathrm{div}}(\calD) := \{ \bm{l} \in \left(\rmL^2(\calD, \rho_\varepsilon)\right)^{n_0}: \widehat{\nabla}\cdot \bm{l} \in \rmL^2(\calD, \rho_\varepsilon) \}$.

Next, we define the following inner product associated with the operator \cref{eq:main}, for $\bm{u},\bm{v}\in\left(\rmL^2(\mathbb{R}^d, \rho_\varepsilon)\right)^{n_0}$ vector-valued functions,  
\begin{equation*}
    \pair*{\bm{u}, \bm{v}}_{\calD} := \int_{\calD} \bm{u}^\transpose \widehat{\mathbf{S}}^{-1} \bm{v} \,\rho_\varepsilon(\rmd z),
\end{equation*}
where $\widehat{\mathbf{S}} = \left(\widehat{\mathbf{A}} + \widehat{\mathbf{A}}^\transpose\right)/2$ is the symmetric part of $\widehat{\mathbf{A}}$. 

Furthermore, we say that a function $u(z)\in \widehat{\rmH}^1(\calD)$ is a weak solution to the equation 
\begin{equation*}
    \L_\varepsilon u= f,
\end{equation*}
if for all $\varphi\in \mathrm{C}^1_c(\calD)$, it holds that
\begin{equation*}
    \int_{\calD} \left(\varepsilon \widehat{\mathbf{A}} \widehat{\nabla} u \cdot \widehat{\nabla} \varphi + u (\bm{b} + \mathbf{B}z) \cdot\nabla\varphi \right)\,\rho_\varepsilon(\rmd z) = \int_{\calD} f\varphi \,\rho_\varepsilon(\rmd z).
\end{equation*}

\begin{assumption}\label{assump:ab}
    Throughout this paper, we assume $\mathcal{A}, \mathcal{B}\subset\mathbb{R}^d$ are two bounded open set with $\mathrm{C}^{2,\alpha}$ boundaries for some $0<\alpha<1$, and $\abs{\partial\mathcal{A}}, \abs{\partial\mathcal{B}}<\infty$. 
    Moreover, $\mathrm{dist}(\mathcal{A}, \mathcal{B})>0$. 
    We also define $\Omega = \left(\overline{\mathcal{A}\cup\mathcal{B}}\right)^\complement$.
\end{assumption}

\subsubsection*{Variational Formulation}
The author is motivated by \cite{AAMN21} and proposes the following form of the variational formulation for \cref{eq:main}: First, we define a functional for \( f \in \rmH^1_{0}(\overline{\calB}^\complement) \) as
\begin{equation}\label{eq:functional}
    J_{\varepsilon, \overline{\calB}^\complement}[f] := \inf_{\bm{g} \in \rmV_{\varepsilon, f}(\overline{\calB}^\complement)} \varepsilon \pair*{\widehat{\mathbf{A}} \widehat{\nabla} f - \bm{g}, \widehat{\mathbf{A}} \widehat{\nabla} f - \bm{g}}_{\overline{\mathcal{B}}^\complement},
\end{equation}
where
\begin{equation*}
    {\rmV_{\varepsilon, f}}(\overline{\calB}^\complement) := \Big\{ \bm{g} \in \widehat{\rmH}^1_{\mathrm{div}} ( \overline{\calB}^\complement ): 
    \varepsilon e^{W/\varepsilon} \widehat{\nabla} \cdot \left( e^{-W/\varepsilon} \bm{g} \right) = - (\bm{b} + \mathbf{B}z) \cdot \nabla f  \Big\}.
\end{equation*}
Similarly, for the adjoint equation, the functional is given by
\begin{equation}\label{eq:functionaladjoint}
    J^\dag_{\varepsilon, \overline{\calB}^\complement}[f] := \inf_{\bm{g} \in \rmV^\dag_{\varepsilon, f}(\overline{\calB}^\complement)} \varepsilon \pair*{\widehat{\mathbf{A}}^\transpose \widehat{\nabla} f - \bm{g}, \widehat{\mathbf{A}}^\transpose \widehat{\nabla} f - \bm{g}}_{\mathcal{B}^\complement},
\end{equation}
where \( \rmV^\dag_{\varepsilon,f} \) is the space of admissible vector fields \( \bm{g} \) that satisfy the adjoint equation
\begin{equation*}
    \varepsilon e^{W/\varepsilon} \widehat{\nabla} \cdot \left( e^{-W/\varepsilon} \bm{g} \right) = \left( \bm{b} + \mathbf{B}z \right) \cdot \nabla f.
\end{equation*}

\subsubsection*{Definition of Capacity} 
To define the capacity, we need to assume certain well-posedness and regularity assumptions on solutions to the Dirichlet problem for \cref{eq:main}.
The following lemma is a direct consequence of the barrier method, see e.g.~\cite{AH24,AP20,Man97,GT01}.
\begin{lemma}\label{thm:bdyreg}
    Let \cref{assump:ab} hold, and let \( \bm{n}_{\Omega} \) be the unit outward normal and \( \widehat{\bm{n}}_{\Omega} \) be the part corresponding to the last \( n_0 \) coordinates.
    Let $\Gamma(\Omega)$ and $\Gamma^\dag(\Omega)$ be the regular parts (where all continuous solutions attain the prescribed boundary values) of the equilibrium and adjoint equilibrium problems, respectively.
    Then the following holds:
    \begin{itemize}
        \item For any \( z \in \partial \Omega \), if \( \widehat{\bm{n}}_\Omega(z) \neq 0 \), then \( z \in \Gamma(\Omega) \) and \( z \in \Gamma^\dag(\Omega) \).
        \item For any \( z \in \partial \Omega \setminus \Gamma(\Omega) \), we have \( \widehat{\bm{n}}_\Omega(z) = 0 \), and \( \mathbf{B}z \cdot \bm{n}_{\Omega}(z) \leq 0 \). Similarly, for \( z \in \partial \Omega \setminus \Gamma^\dag(\Omega) \), we have \( \widehat{\bm{n}}_\Omega(z) = 0 \) and \( \mathbf{B}z \cdot \bm{n}_{\Omega}(z) \geq 0 \).
        \item \( \Gamma(\Omega) \cup \Gamma^\dag(\Omega) \) differs from \( \partial \Omega \) only on a set of surface measure zero.
        \item If \( \mathbf{A} \) is uniformly elliptic, then \( \Gamma(\Omega) = \Gamma^\dag(\Omega) = \partial \Omega \).
    \end{itemize}
\end{lemma}  
\begin{definition}
    Let $\calD\subset\R^d$ be an open set with $\mathrm{C}^{2,\alpha}$ boundary, and let $f\in \mathrm{C}(\overline{\calD})\cap \mathrm{C}^1(\calD)$. 
    We define the $\widehat{\mathbf{A}}$-normal derivative of $f$ at $z_0\in\partial\calD$, denoted by
    \begin{equation*}
        \widehat{\mathbf{A}}(z_0)\widehat{\nabla}f(z_0)\cdot\widehat{\bm{n}}_{\calD}(z_0) := \lim_{\delta\to 0} \widehat{\mathbf{A}}(z_0 - \delta\widehat{\bm{n}}_{\calD})\widehat{\nabla}f(z_0 - \delta\widehat{\bm{n}}_{\calD}) \cdot \widehat{\bm{n}}_{\calD} (z_0)
    \end{equation*}
    where $\bm{n}_{\calD}$ is the unit outward normal to $\calD$, and $\widehat{\bm{n}}_{\calD}$ is the part corresponding to the last $n_0$ coordinates.
    Similarly, we can define the $\widehat{\mathbf{A}}^\transpose$-normal derivative of $f$.
\end{definition}
\begin{assumption}\label{assump:equilibrium}
    Throughout the paper, we assume that there exists a unique weak solution \( h_\varepsilon = h_{\varepsilon;\calA,\calB} \in \mathrm{C}(\overline{\Omega})\cap \mathrm{C}^1(\Omega) \) to the following Dirichlet problem, known as the equilibrium problem:
    \begin{equation}\label{eq:equilibriumproblem}
        \begin{cases}
            -\L_\varepsilon h_{\varepsilon} = 0, & \text{in } \Omega, \\
            \hfill h_{\varepsilon} = 1, & \text{on } \Gamma(\Omega)\cap\partial\calA, \\
            \hfill h_{\varepsilon} = 0, & \text{on } \Gamma(\Omega)\cap\partial\calB,
        \end{cases}
    \end{equation}
    where $\Gamma(\Omega)$ denote the regular part of the boundary.
    Moreover, we assume there exists uniformly bounded $\widehat{\mathbf{A}}$-normal derivative of \( h_\varepsilon \) on $\partial\Omega$.

    Similarly, we denote by \( h_\varepsilon^\dag \in \mathrm{C}(\overline{\Omega})\cap \mathrm{C}^1(\Omega) \) the weak solution to the equilibrium problem for the adjoint operator \cref{eq:adjoint}, attains the boundary value on $\Gamma^\dag(\Omega)$, and also satisfies the $\widehat{\mathbf{A}}$-normal derivative condition.
\end{assumption}

\begin{remark}\label{rmk:reg}
    When $\mathbf{A}$ is uniformly elliptic and H\"older continuous, the well-posedness and regularity are guaranteed by classical elliptic theory. In particular \( h_\varepsilon \in \mathrm{C}^{1,\alpha}(\overline{\Omega}) \) by \cite[Thm.~8.34]{GT01}, and the boundary value is attained on the entire boundary since \( \partial\Omega \) is regular under the cone condition. Hence, the $\widehat{\mathbf{A}}$-normal derivative coincides with the classical directional derivative.

    However, in the degenerate case, particularly in unbounded non-product domains, to the best of our knowledge, there exists a Perron's solution for \cref{eq:equilibriumproblem} that is continuous up to the boundary and locally a weak solution, as shown in \cite{AH24}. 
    It is important to note that the solution will most likely not be \( \mathrm{C}^{1,\alpha}(\overline{\Omega}) \) no matter how regular the coefficients are, since there are transport directions which lacks boundary regularity (e.g.~\cite[Appendix]{Zhu22}). 
    Further regularity of the equilibrium potential will be addressed in future work.
\end{remark}
\begin{remark}
    After extending $h_\varepsilon = 1$ in $\calA$, it is indeed a lower semi-continuous function, and also a weak supersolution, namely for $\varphi\in C_c^\infty(\overline{\calB}^\complement)$, we have 
    \begin{equation*}
        \int_{\overline{\calB}^\complement} \left(\varepsilon \widehat{\mathbf{A}}\widehat{\nabla} h \cdot\widehat{\nabla}\varphi + h_\varepsilon \left(\bm{b} + \mathbf{B}z\right) \cdot\nabla \varphi\right) \,\rho_\varepsilon(\rmd z) \geq 0,
    \end{equation*}
    as computed in \cref{sec:capdef}.
\end{remark}
The following weak maximum principle is known for continuous weak solutions, see \cite{AP20,AH24}.
\begin{lemma}\label{thm:max}
    Let \cref{assump:ab} hold, if $f\in \mathrm{C}(\overline{\Omega})\cap \mathrm{C}^1(\Omega)$ is a weak solution to $\L_\varepsilon f = 0$ in $\Omega$, then it satisfies the maximum principle, namely
    \begin{equation*}
        \min_{y\in\Gamma(\Omega)} f(y) \leq f(z) \leq \max_{y\in\Gamma(\Omega)} f(y),\quad \forall z\in\Omega.
    \end{equation*}
\end{lemma}

\begin{definition}[Capacity]\label{def:cap}
    Let \cref{assump:ab,assump:equilibrium} hold. The capacity and the adjoint capacity between \( \calA \) and \( \calB \) are defined as
    \begin{equation*}
        \begin{split}
            \capa_{\varepsilon} (\calA, \calB) &= \frac{1}{Z_\varepsilon} \int_{\partial \calA} \left( \varepsilon (\widehat{\mathbf{A}} \widehat{\nabla} h_{\varepsilon}) \cdot \widehat{\bm{n}}_{\Omega} + (h_\varepsilon - 1) (\bm{b} + \mathbf{B}z) \cdot \bm{n}_{\Omega} \right) e^{-W/\varepsilon} \sigma(\rmd z), \\
            \capa_{\varepsilon}^\dag (\calA, \calB) &= \frac{1}{Z_\varepsilon} \int_{\partial \calA} \left( \varepsilon (\widehat{\mathbf{A}}^\transpose \widehat{\nabla} h_{\varepsilon}^\dag) \cdot \widehat{\bm{n}}_{\Omega} + (1 - h_\varepsilon^\dag) (\bm{b} + \mathbf{B}z) \cdot \bm{n}_{\Omega} \right) e^{-W/\varepsilon} \sigma(\rmd z),
        \end{split}
    \end{equation*}
    where \( \bm{n}_{\Omega} \) is the unit outward normal vector to \( \Omega \), \( \widehat{\bm{n}}_{\Omega} \) corresponds to the part related to the last \( n_0 \) coordinates, and \( \sigma(\rmd z) \) is the surface measure on \( \partial \calA \).
\end{definition}
The capacity is well-defined and is positive due to \cref{thm:max}.
Note that the second part of the integral represents the potential jump discontinuity of $h_\varepsilon,h_\varepsilon^\dag$ on the irregular part of the boundary.
We motivate this definition of capacity in \cref{sec:capdef}.

We remark that different definitions of capacity can be found in the literature provided the existence of a Green function, for example \cite{KLT18} for evolutionary equations $\partial_t u = \L_\varepsilon u$.
\begin{remark}
    If $\mathbf{A}$ is non-degenerate and H\"older continuous, from \cref{rmk:reg,thm:bdyreg}, we indeed have
    \begin{equation*}
        \capa_\varepsilon(\calA, \calB) = \varepsilon \int_{\partial \calA} (\mathbf{A} \nabla h_{\varepsilon}) \cdot \bm{n}_{\Omega} e^{-W/\varepsilon} \sigma(\rmd x),
    \end{equation*}
    which is the classical definition of capacity.
\end{remark}

\subsection{Main Results}
An important observation is that the capacity is symmetric.
\begin{lemma}\label{thm:capsymm}
    Let \cref{assump:ab,assump:equilibrium} hold. 
    Then, the following equality holds:
    \begin{equation*}
        \capa_{\varepsilon}(\mathcal{A}, \calB) = \capa_{\varepsilon}^\dag (\mathcal{A}, \calB).
    \end{equation*}
\end{lemma}

\subsubsection{Variational Capacity}
Our first result relates the variational formulation to the capacity. 

\begin{theorem}[Variational Capacity]\label{thm:variationalcap}
    Let \cref{assump:ab,assump:equilibrium} hold. Then, the following relation holds:
    \begin{equation}\label{eq:variationalcap}
        \varepsilon \pair*{ \widehat{\mathbf{S}}\widehat{\nabla} h_\varepsilon,\widehat{\mathbf{S}} \widehat{\nabla} h_\varepsilon }_{\Omega}  \leq \capa_{\varepsilon} (\calA, \calB) \leq \inf_{f\in \rmH^1_{ 0}(\overline{\calB}^\complement),\,f\geq \mathcal{X}_{\calA}} J_{\overline{\calB}^\complement} [f].    
    \end{equation}
    In particular, if $\mathbf{A}$ is non-degenerate and H\"older continuous, then the equality holds, 
    and the minimizer for the variational formulation is attained at \( \bar{f} = (h_\varepsilon + h_\varepsilon^\dag)/2 \), and \( \bar{g} = (\mathbf{A}\nabla h_\varepsilon - \mathbf{A}^\transpose \nabla h_\varepsilon^\dag)/2 \).

    Similar results hold for the adjoint capacity \( \capa_\varepsilon^\dag(\calA,\calB) \) as well.
\end{theorem}
The proof is by combining \cref{thm:caplower,thm:capupper}.
\begin{remark}
    In \cite{AAMN21}, the infimum is taken over a ``hypoelliptic space'' which essentially contains distributions and is not handy for practical computations.
    The choice of the space \( \rmH^1_{0}(\overline{\calB}^\complement) \) will likely preserve the infimum but exclude the minimizer.
\end{remark}
\begin{remark}
    The right-hand side of \cref{eq:variationalcap} corresponds to an obstacle problem. However, as shown in the theorem, the minimizer does not solve the equilibrium problem \cref{eq:equilibriumproblem}; instead, it is the average of the two adjoint equilibrium solutions.
\end{remark}
\begin{remark}
    An important question remains open here is that for the degenerate case, when does the equality hold in \cref{eq:variationalcap}? 
    The author believes that $h_\varepsilon = 1$ on $\partial\calA$ p.p.~(differs only on a set of capacity zero) is a necessary condition for the equality to hold as hinted from \cref{thm:caplower} and will be addressed in future work.
\end{remark}
}

\subsubsection{Quantitative Capacity Estimates for the Elliptic Case}
Throughout this section, we assume the following.

\begin{assumption}\label{assump:elliptic}
    Let \( \mathbf{A} \) be uniformly elliptic on \( \mathbb{R}^d \) and Hölder continuous of order \( \alpha \). 
    We also set \( \mathbf{B}=0 \), and \( \mathbf{b}(z) = -\bm{l}(z) \) for some smooth vector field \( \bm{l}: \mathbb{R}^d \to \mathbb{R}^d \), such that \cref{assump:divfree} holds.
\end{assumption}
The operator now becomes:
\begin{equation*}
    \L_\varepsilon f = \varepsilon e^{W/\varepsilon} \nabla\cdot\left[ e^{-W/\varepsilon}\left(\mathbf{A}\nabla f - {f\bm{l}}/{\varepsilon}\right) \right].
\end{equation*}

The first application of the variational capacity is the following rough upper and lower bounds on capacities. 
This result is known and can be proven using a non-variational approach presented in \cite{LS22}. 
However, we propose a new proof that utilizes the non-self-adjoint variational capacity.
\begin{theorem}\label{thm:roughbounds}
    Let \( \mathcal{A} \) be an open set, and let \( x \in \overline{\mathcal{A}}^\complement \). Denote by \( z^\star = z^\star(x, \overline{\mathcal{A}}) \) a point such that 
    \begin{equation*}
        W(z^\star) = \inf_{\gamma(t)} \sup_{t\in[0,1]} W(\gamma(t)),
    \end{equation*}
    where the infimum is taken over all continuous paths from \( x \) to \( \overline{\mathcal{A}} \).
    Then, there exist constants \( c_1 \) and \( c_2 \) independent of \( \varepsilon \) such that 
    \begin{equation*}
        \capa_\varepsilon(B_\varepsilon(z), \overline{\mathcal{A}}) \geq c_1 \varepsilon^d e^{-W(z^\star)/\varepsilon}, \quad \text{and} \quad
        \capa_\varepsilon (B_\varepsilon(z), \overline{\mathcal{A}}) \leq c_2 \varepsilon^{d-4} e^{-W(z^\star)/\varepsilon}.
    \end{equation*}
\end{theorem} 

\begin{assumption}\label{assump:twowell}
    Let \( W \) be a smooth Morse function with two local minima \( \bm{m}_0 \) and \( \bm{m}_1 \), and without loss of generality, let \( 0 \) be the unique saddle point. We also set \( \mathcal{A} = B_\varepsilon(\bm{m}_1) \) and \( \mathcal{B} = B_\varepsilon(\bm{m}_0) \).
\end{assumption}
We write \( H := W(0),\, h_0 := W(\bm{m}_0),\, \text{and } h_1 := W(\bm{m}_1) \).
Next, denote by \( \mathbf{H}_z = \nabla^2 W(z) \) the Hessian of \( W \) and \( \mathbf{L}_z = \nabla \bm{l}(z) \) the Jacobian of \( \bm{l} \) at the point \( z \in \mathbb{R}^d \).
It will be shown in \cref{thm:hml} that the matrix \( \mathbf{H}_{0}\mathbf{A}_{0} + \mathbf{L}^\transpose_{0} \) has a unique negative eigenvalue, denoted by \( -\mu \).

\begin{theorem}\label{thm:capestimate}
    Let \cref{assump:ab,assump:elliptic,assump:twowell} hold. 
    Then, the following capacity estimate holds:
    \begin{equation}
        \capa_\varepsilon(B_\varepsilon(\bm{m}_1), B_\varepsilon(\bm{m}_0)) = (1 + o_\varepsilon(1)) \frac{1}{Z_\varepsilon} \frac{(2\pi\varepsilon)^{d/2}}{2\pi} \frac{\mu}{\sqrt{-\det(\mathbf{H}_0)}} e^{-H/\varepsilon}.
    \end{equation}
\end{theorem}

\begin{remark}
    The \( o_\varepsilon(1) \) term now depends on the Hölder continuity exponent \( \alpha \) of \( \mathbf{A} \), specifically in a neighborhood of the saddle point, and is of the order \( \varepsilon^{\alpha/2} |\log \varepsilon|^{\alpha/2} \). 
\end{remark}

As a corollary, the capacity estimate yields the Eyring-Kramers formula.

\begin{theorem}
    Let \cref{assump:ab,assump:elliptic,assump:twowell} hold, and let \( w \) solve \cref{eq:landscape}. 
    Then, the following Eyring-Kramers formula holds:
    \begin{equation*}
        w(\bm{m}_1) = \mathbb{E}^{\bm{m}_1}[\tau_{B_\varepsilon(\bm{m}_0)}] = \left(1 + o_\varepsilon(1)\right) \frac{2\pi}{\mu} \sqrt{\frac{-\det\mathbf{H}_0}{\det \mathbf{H}_{\bm{m}_1}}} e^{(H-h_1)/\varepsilon}.
    \end{equation*}
\end{theorem}
The proof of the Eyring-Kramers formula is standard, given the capacity estimate in \cref{thm:capestimate} and the rough bounds on the equilibrium potential in \cref{thm:roughbounds}. 
This can be found in \cite{BEGK04,LMS19,LS22} and is omitted.
This result generalizes the findings from \cite{LMS19,LS22} and coincides with the result from \cite{BLPM22}.

\subsection{Outline of the paper}
The paper is organized as follows:
In \cref{sec:varitionalcap}, we motivate the definition of capacity and prove its symmetry as well as the variational formula in \cref{thm:capsymm} and \cref{thm:variationalcap}.
In \cref{sec:roughestimate}, we establish the rough bounds on the capacity, as stated in \cref{thm:roughbound}, under \cref{assump:elliptic}.
Finally, in \cref{sec:sharpestimate}, we derive the sharp capacity estimate in \cref{thm:capestimate} under \cref{assump:elliptic}, following the ideas presented in \cite{LS22}.

\section{The Variational Capacity}\label{sec:varitionalcap}
In this section, we first motivate the definition of capacity (see \cref{def:cap}) and then prove \cref{thm:variationalcap}. 

{
\subsection{Motivation for the Capacity}\label{sec:capdef}
Before justifying the definition of capacity, we first prove the following integration by part formulas.
\begin{lemma}\label{thm:intbyparts}
    Let $\calD\subset\R^{d}$ be an open set with $\mathrm{C}^{2,\alpha}$ boundary, and let $f\in \mathrm{C}^1(\calD)\cap \mathrm{C}(\overline{\calD})$ be a weak solution to the Dirichlet problem for \cref{eq:main} such that there exists uniformly bounded $\widehat{\mathbf{A}}$-nomral derivative on $\partial\calD$. 
    Then, for any $\varphi \in \mathrm{C}_c^1(\overline{\calD})$, we have
    \begin{multline*}
        \int_{\calD} \left(\varepsilon \widehat{\mathbf{A}}\widehat{\nabla} f \cdot\widehat{\nabla}\varphi + f \left(\bm{b} + \mathbf{B}z\right) \cdot\nabla \varphi\right) \,\rho_\varepsilon(\rmd z) \\
        = \frac{1}{Z_\varepsilon}\int_{\partial\calD} \left(\varepsilon\widehat{\mathbf{A}}\widehat{\nabla}f\cdot\widehat{\bm{n}}_{\calD}  + f \left(\bm{b} - \mathbf{B}z\right)\cdot\bm{n}_{\calD} \right)\varphi\, e^{-W/\varepsilon} \sigma(\rmd z),
    \end{multline*}
    where $\bm{n}_{\calD}$ is the unit outward normal to $\calD$, $\widehat{\bm{n}}_{\calD}$ is the part corresponding to the last $n_0$ coordinates, and $\sigma(\rmd z)$ is the surface measure on $\partial\calD$.
\end{lemma}
\begin{proof}
    The proof is standard provided there is enough regularity on the boundary and the solution $f$.
    Indeed, we choose a test function $\psi = \varphi \phi_\delta$ where $\phi_\delta$ is a smooth cutoff function that is 1 on $\calD_{-\delta}:=\{z\in\calD:\mathrm{dist}(z,\partial\calD)>\delta\} $ and 0 on $\calD_{-\delta/2}$ which is defined analogously.
    Plugging in $\psi$ into the weak formulation of \cref{eq:main}, we obtain 
    \begin{multline*}
        \int_{\calD} \left(\varepsilon \widehat{\mathbf{A}}\widehat{\nabla} f \cdot\widehat{\nabla}\varphi + f \left(\bm{b} + \mathbf{B}z\right) \cdot\nabla \varphi\right)\phi_\delta \,\rho_\varepsilon(\rmd z) \\
        + \int_{\calD} \left(\varepsilon \widehat{\mathbf{A}}\widehat{\nabla} f \cdot\widehat{\nabla}\phi_\delta + f \left(\bm{b} + \mathbf{B}z\right) \cdot\nabla \phi_\delta\right)\varphi \,\rho_\varepsilon(\rmd z) = 0.
    \end{multline*}
    Since $\partial\calD$ is smooth enough, $\nabla\phi_{\delta}$ can be chosen to approximate $-\bm{n}_{\calD}$, and the result follows by taking the limit as $\delta\to 0$.
\end{proof}

In classical potential theory, capacity describes the amount of charge induced on a capacitor, with the charge concentrating on the surface of the capacitor. 
The solution $h_\varepsilon$ to the equilibrium problem \cref{eq:equilibriumproblem} represents the potential distribution, with potential 1 on $\calA$ and 0 on $\calB$.  
By extending $h_\varepsilon$ to be 1 on $\calA$ and 0 on $\overline{\calB}$, 
it follows that $h_\varepsilon$ is lower semi continuous and $-\L_\varepsilon h_\varepsilon$ is a surface measure (known as the harmonic measure for the Laplace equation) on $\partial\calA$ and represents the distribution of charges. 
To see this, we take a test function $\varphi\in \mathrm{C}_c^1(\calB^\complement)$ such that $\varphi = 0$ on $\Gamma^\dag(\overline{\calB}^\complement)$, and compute in the weak sense. 
Using integration by parts \cref{thm:intbyparts} and \cref{thm:bdyreg}, 
we derive
\begin{align*}
    &\int_{\overline{\calB}^\complement} -\L_\varepsilon h_\varepsilon \varphi \,\rho_\varepsilon(\rmd z) := \int_{\overline{\calB}^\complement} \left(\varepsilon \widehat{\mathbf{A}}\widehat{\nabla} h \cdot\widehat{\nabla}\varphi + h_\varepsilon \left(\bm{b} + \mathbf{B}z\right) \cdot\nabla \varphi\right) \,\rho_\varepsilon(\rmd z)
    \\
    &= \int_{\Omega} \left(\varepsilon \widehat{\mathbf{A}}\widehat{\nabla} h \cdot\widehat{\nabla}\varphi + h_\varepsilon \left(\bm{b} + \mathbf{B}z\right) \cdot\nabla \varphi\right) \,\rho_\varepsilon(\rmd z)
    + \int_{\calA} \left(\bm{b} +\mathbf{B}z\right)\cdot\nabla \varphi \,\rho_\varepsilon(\rmd z)
    \\
    &= \frac{1}{Z_\varepsilon}\int_{\partial\Omega} \left(\varepsilon\widehat{\mathbf{A}}\widehat{\nabla}h_\varepsilon\cdot\widehat{\bm{n}}_{\Omega}  + h_\varepsilon \left(\bm{b} - \mathbf{B}z\right)\cdot\bm{n}_{\Omega} \right)\varphi \,e^{-W/\varepsilon} \sigma(\rmd z)
    + \frac{1}{Z_\varepsilon}\int_{\partial\calA} \varphi \left(\bm{b} + \mathbf{B}z\right) \bm{n}_{\calA}  e^{-W/\varepsilon}\sigma(\rmd z).
\end{align*}
Using the fact from \cref{thm:bdyreg} that $\widehat{n}_{\Omega}\varphi = 0$ and $h_\varepsilon \varphi = 0$ on $\partial\calB$ we get
\begin{equation*}
    \int_{\overline{\calB}^\complement} -\L_\varepsilon h_\varepsilon \varphi \,\rho_\varepsilon(\rmd z) = \frac{1}{Z_\varepsilon} \int_{\partial\calA} \left(\varepsilon(\widehat{\mathbf{A}} \widehat{\nabla} h_{\varepsilon})\cdot \widehat{\bm{n}}_{\Omega} +  (h_\varepsilon - 1) \left(\bm{b} + \mathbf{B}z\right)\cdot\bm{n}_{\Omega}   \right) \varphi \,e^{-W/\varepsilon}  \sigma(\rmd z).
\end{equation*}
Therefore, the definition in \cref{def:cap} is justified.

The previous calculation shows that $-\L_\varepsilon h_\varepsilon$ is a measure on $\partial\calA$ and similarly for $-\L_\varepsilon^\dag h_\varepsilon^\dag$. 
Now, let $w$ be the landscape function defined in \cref{eq:landscape}. 
Following the approach in \cite{BEGK04}, we derive 
\begin{equation*}
    \int_{\overline{\calB}^\complement} h_\varepsilon^\dag \,\rho_\varepsilon(\rmd z) = \int_{\overline{\calB}^\complement} (-\L_\varepsilon w) h_\varepsilon^\dag \,\rho_\varepsilon(\rmd z) = \int_{\overline{\calB}^\complement} w (-\L_\varepsilon^\dag h_\varepsilon^\dag) \,\rho_\varepsilon(\rmd z) \approx \sup_{z\in\partial\calA} w(z)\times \capa_\varepsilon(\calA,\calB),
\end{equation*}
where $\approx$ depends on the Harnack inequality for the equation \cref{eq:landscape}, see e.g.~\cite{AR22}.
This again validates our definition of capacity and demonstrates its usefulness for metastability estimates involving the general Fokker-Planck-Kolmogorov operator \cref{eq:main}.

}

\begin{delayedproof}{thm:capsymm}
    By the definition of the capacity in \cref{def:cap}, and noting that $h_\varepsilon^\dag = 1$ on $\Gamma^\dag(\Omega)$, we can introduce it into the integral:
    \begin{equation*}
        \capa_\varepsilon(\calA, \calB) = \frac{1}{Z_\varepsilon} \int_{\partial\calA} h_\varepsilon^\dag \times \left(\varepsilon(\widehat{\mathbf{A}} \widehat{\nabla} h_{\varepsilon})\cdot \widehat{\bm{n}}_{\Omega} +  (h_\varepsilon - 1) \left(\bm{b} + \mathbf{B}z\right)\cdot\bm{n}_{\Omega}   \right) \,e^{-W/\varepsilon}  \sigma(\rmd z).
    \end{equation*}
    The lemma then follows by integration by parts, \cref{thm:intbyparts}.
\end{delayedproof}

\subsection{Proof of the Variational Capacity}
To prove \cref{thm:variationalcap}, we prove the two inequalities in \cref{eq:variationalcap} separately in the following two propositions.
\begin{proposition}\label{thm:caplower}
    Let \cref{assump:ab,assump:equilibrium} hold. 
    Then the following inequality holds:
    \begin{equation*}
        \varepsilon \pair*{ \widehat{\mathbf{S}}\widehat{\nabla} h_\varepsilon, \widehat{\mathbf{S}} \widehat{\nabla} h_\varepsilon }_{\Omega} \leq \capa_\varepsilon(\calA,\calB).
    \end{equation*}
    Moreover, if $\Gamma(\Omega) = \partial\Omega$ a.e., then 
    \begin{equation*}
        \varepsilon \pair*{ \widehat{\mathbf{S}}\widehat{\nabla} h_\varepsilon,\widehat{\mathbf{S}} \widehat{\nabla} h_\varepsilon }_{\Omega}  = \capa_\varepsilon(\calA,\calB).
    \end{equation*}
    Similar results hold for $\capa_\varepsilon^\dag(\calA,\calB)$ as well. 
\end{proposition}

\begin{proof}
    To prove the result, we first introduce $h_\varepsilon$ into the definition of capacity (see \cref{def:cap}). 
    However, we can only introduce it into the derivative term, namely:
    \begin{equation*}
        \capa_\varepsilon(\calA,\calB) = \frac{1}{Z_\varepsilon}\int_{\partial\calA}  \left(h_\varepsilon \times \varepsilon\widehat{\mathbf{A}} \widehat{\nabla} h_\varepsilon \cdot \widehat{\bm{n}}_\Omega  +  (h_\varepsilon - 1)\left(\bm{b} + \mathbf{B}z\right)\cdot \bm{n}_{\Omega}\right) \, e^{-W/\varepsilon} \sigma(\rmd z). 
    \end{equation*} 
    By using integration by parts \cref{thm:intbyparts}, we get:
    \begin{align*}
        \capa_\varepsilon(\calA,\calB) = &\int_{\Omega} \varepsilon \widehat{\mathbf{A}}\widehat{\nabla} h_\varepsilon \cdot \widehat{\nabla} h_\varepsilon \,\rho_\varepsilon(\rmd z) - \int_{\Omega} h_\varepsilon \left(\bm{b} + \mathbf{B}z\right) \cdot \nabla h_\varepsilon \,\rho_\varepsilon(\rmd z) \\
        &+ \int_{\partial \calA} (h_\varepsilon - 1)\left(\bm{b} + \mathbf{B}z\right) \bm{n}_{\Omega} e^{-W/\varepsilon} \sigma(\rmd z)/Z_\varepsilon \\
        = & \int_{\Omega} \varepsilon \widehat{\mathbf{A}}\widehat{\nabla} h_\varepsilon \cdot \widehat{\nabla} h_\varepsilon \,\rho_\varepsilon(\rmd z) \underbrace{- \int_{\partial\calB} \frac{h_\varepsilon^2}{2} \left(\bm{b} + \mathbf{B}z\right) \cdot \bm{n}_\Omega \, e^{-W/\varepsilon} \sigma(\rmd z) / Z_\varepsilon}_{\mathrm{I}} \\
        &+ \underbrace{\int_{\partial \calA} \left(h_\varepsilon - 1 - \frac{h_\varepsilon^2}{2}\right) \left(\bm{b} + \mathbf{B}z\right) \bm{n}_{\Omega} e^{-W/\varepsilon} \sigma(\rmd z)/Z_\varepsilon}_{\mathrm{II}}.
    \end{align*}
    
    First, observe that for \( z \in \partial\Omega \setminus \Gamma(\Omega) \), we must have \( \widehat{n}_\Omega(z) = 0 \) and \( \mathbf{B}z \cdot \bm{n}_{\Omega}(z) \leq 0 \) from \cref{thm:bdyreg}.
    Using \cref{assump:coeff}, it follows that:
    \begin{equation*}
        \mathrm{I} = - \int_{\partial\calB \setminus \Gamma(\Omega)} \frac{h_\varepsilon^2}{2} \left(\bm{b} + \mathbf{B}z\right) \cdot \bm{n}_\Omega \, e^{-W/\varepsilon} \sigma(\rmd z) / Z_\varepsilon = - \int_{\partial\calB \setminus \Gamma(\Omega)} \frac{h_\varepsilon^2}{2} \mathbf{B}z \cdot \bm{n}_\Omega \, e^{-W/\varepsilon} \sigma(\rmd z) / Z_\varepsilon \geq 0.
    \end{equation*}
    
    For \( \mathrm{II} \), we use a similar approach to derive:
    \begin{multline*}
        \mathrm{II} = \int_{\partial \calA \cap \Gamma(\Omega)}  - \frac{1}{2} \left(\bm{b} + \mathbf{B}z\right) \bm{n}_{\Omega} e^{-W/\varepsilon} \sigma(\rmd z)/Z_\varepsilon \\
        + \int_{\partial \calA \setminus \Gamma(\Omega)} \left(h_\varepsilon - 1 - \frac{h_\varepsilon^2}{2}\right) \left(\bm{b} + \mathbf{B}z\right) \bm{n}_{\Omega} e^{-W/\varepsilon} \sigma(\rmd z)/Z_\varepsilon.
    \end{multline*}
    
    By the maximum principle, \cref{thm:max}, we have \( 0 \leq h_\varepsilon \leq 1 \). Hence, \( -1 \leq h_\varepsilon - 1 - h_\varepsilon^2/2 \leq -1/2 \). It follows that
    \begin{equation*}
        \mathrm{II} \geq \frac{1}{2} \int_{\partial\calA} (\bm{b} + \mathbf{B}z) \cdot \bm{n}_{\Omega} \, e^{-W/\varepsilon} \sigma(\rmd z)/Z_\varepsilon \geq 0,
    \end{equation*}
    where the surface integral vanishes due to \cref{assump:divfree}. 

    Combining the estimates for \( \mathrm{I} \) and \( \mathrm{II} \), we obtain:
    \begin{equation*}
        \capa_\varepsilon(\calA, \calB) \geq \int_{\Omega} \varepsilon \widehat{\mathbf{A}} \widehat{\nabla} h_\varepsilon \cdot \widehat{\nabla} h_\varepsilon \, \rho_\varepsilon(\rmd z) = \varepsilon \int_{\Omega} \widehat{\mathbf{S}} \widehat{\nabla} h_\varepsilon \cdot \widehat{\nabla} h_\varepsilon \, \rho_\varepsilon(\rmd z) = \varepsilon \pair*{ \widehat{\mathbf{S}} \widehat{\nabla} h_\varepsilon, \widehat{\mathbf{S}} \widehat{\nabla} h_\varepsilon }_{\Omega}.
    \end{equation*}
    This completes the proof of the inequality. 

    Finally, if $\Gamma(\Omega) = \partial\Omega$ a.e., it is obvious that the term $\mathrm{I}$ and $\mathrm{II}$ vanishes which proves the equality.
\end{proof}

{
Next, we show the right inequality in \cref{thm:variationalcap}.
\begin{proposition}\label{thm:capupper}
    Let \cref{assump:ab,assump:equilibrium} hold. 
    Then the following inequality holds true
    \begin{equation}\label{eq:capupper}
        \capa_{\varepsilon} (\calA, \calB) \leq \inf_{f\in \rmH^1_{ 0}(\overline{\calB}^\complement), \, f\geq \mathcal{X}_{\calA}} J_{\varepsilon,\overline{\calB}^\complement} [f].
    \end{equation}
    Moreover, if $\mathbf{A}$ is non-degenerate and H\"older continuous, then 
    \begin{equation*}
        \capa_{\varepsilon}(\mathcal{A}, \calB) = \inf_{f\in \rmH^1_{ 0}(\overline{\mathcal{B}}^\complement),\, f\geq \mathcal{X}_\calA} J_{\varepsilon,\overline{\calB}^\complement}[f],
    \end{equation*}
    and the minimum is attained at $\bar{f} = (h_\varepsilon + h_\varepsilon^\dag)/2$, and $\bar{g} = (\widehat{\mathbf{A}}\widehat{\nabla} h_\varepsilon - \widehat{\mathbf{A}}^\transpose \widehat{\nabla} h_\varepsilon^\dag)/2$.
\end{proposition}
\begin{proof}
    We begin by proving the lower bound for the variational formula \cref{eq:capupper}.

    First, we prove that for any $f\in \rmH^1_0(\overline{\calB}^\complement)$ with $f \geq \mathcal{X}_\calA$, we have 
    \begin{equation}\label{eq:capineq}
        \varepsilon \pair*{ \widehat{\mathbf{A}} \widehat{\nabla} f - \bm{g}, \widehat{\mathbf{S}} \widehat{\nabla} h_\varepsilon^\dag }_{\overline{\calB}^\complement} \geq \capa_{\varepsilon}^\dag(\calA, \calB) = \capa_\varepsilon(\calA, \calB).
    \end{equation}
    Indeed, by the definition of $\bm{g}$ in \cref{eq:functional}, and using integration by parts, \cref{thm:intbyparts}, along with the fact that $f h_\varepsilon^\dag = 0$ on $\partial \calB$ by \cref{thm:bdyreg}, we compute:
    \begin{align*}
        & \varepsilon \pair*{ \widehat{\mathbf{A}} \widehat{\nabla} f - \bm{g}, \widehat{\mathbf{S}} \widehat{\nabla} h_\varepsilon^\dag }_{\overline{\calB}^\complement} = \varepsilon \int_{\overline{\calB}^\complement} \left( \widehat{\mathbf{A}} \widehat{\nabla} f - \bm{g} \right) \cdot \widehat{\nabla} h_\varepsilon^\dag \, \rho_\varepsilon(\rmd z) \\
        &= \varepsilon \int_{\Omega} \widehat{\mathbf{A}} \widehat{\nabla} f \cdot \widehat{\nabla} h_\varepsilon^\dag \, \rho_\varepsilon(\rmd z) - \int_{\overline{\calB}^\complement} \left( (\bm{b} + \mathbf{B}z) \cdot \nabla f \right) h_\varepsilon^\dag \, \rho_\varepsilon(\rmd z) 
        \\
        &= \varepsilon \int_{\Omega} \widehat{\mathbf{A}} \widehat{\nabla} f \cdot \widehat{\nabla} h_\varepsilon^\dag \, \rho_\varepsilon(\rmd z) - \int_{\Omega} \left( (\bm{b} + \mathbf{B}z) \cdot \nabla f \right) h_\varepsilon^\dag \, \rho_\varepsilon(\rmd z)
        + \int_{\partial \calA} f (\bm{b} + \mathbf{B}z) \cdot \bm{n}_{\calA} \, e^{-W/\varepsilon} \sigma(\rmd z)
        \\
        &= \frac{1}{Z_\varepsilon} \int_{\partial \calA} f \left( \varepsilon \widehat{\mathbf{A}}^\transpose \widehat{\nabla} h_\varepsilon^\dag \cdot \widehat{\bm{n}}_{\Omega} + (1 - h_\varepsilon^\dag) (\bm{b} + \mathbf{B}z) \cdot \bm{n}_\Omega \right) e^{-W/\varepsilon} \sigma(\rmd z).
    \end{align*}
    Thus, using $f \geq \mathcal{X}_\calA$ yields \cref{eq:capineq}. 

    In particular, if $f \equiv 1$ on $\partial \calA$, we obtain:
    \begin{equation*}
        \varepsilon \pair*{ \widehat{\mathbf{A}} \widehat{\nabla} f - \bm{g}, \widehat{\mathbf{S}} \widehat{\nabla} h_\varepsilon^\dag }_{\overline{\calB}^\complement} = \capa_{\varepsilon}^\dag(\calA, \calB) = \capa_\varepsilon(\calA, \calB).
    \end{equation*}
    This observation is crucial in both \cite{LMS19, LS22}. 

    To prove \cref{eq:capupper}, we simply apply Cauchy's inequality to \cref{eq:capineq} and use \cref{thm:caplower}. 
    
    Furthermore, if $\mathbf{A}$ is non-degenerate and H\"older continuous, then $\Gamma(\Omega) = \partial\Omega$ by \cref{thm:bdyreg}. 
    In addition, $h_\varepsilon,h_\varepsilon^\dag\in \rmH^1(\Omega)$ and $h_\varepsilon,h_\varepsilon^\dag \equiv 1$ on $\partial\calA$.
    By extending $h_\varepsilon,h_\varepsilon^\dag$ to be 1 on $\calA$, they belong to $\rmH^1_0(\overline{\calB}^\complement)$ which are valid candidates for the variational formula.
    We claim that the minimum is achieved by $\bar{f} = (h_\varepsilon + h_\varepsilon^\dag)/2$ and $\bar{\bm{g}} = (\widehat{\mathbf{A}} \widehat{\nabla} h_\varepsilon - \widehat{\mathbf{A}}^\transpose \widehat{\nabla} h_\varepsilon^\dag)/2$. 
    It is easy to verify by definition that $\bar{\bm{g}} \in \rmV_{\varepsilon, \overline{\calB}^\complement}[\bar{f}]$, namely,
    \begin{equation*}
        \varepsilon e^{W/\varepsilon} \widehat{\nabla} \cdot \left[ e^{-W/\varepsilon} \bar{\bm{g}} \right] = - (\bm{b} + \mathbf{B}z) \cdot \nabla \bar{f}.
    \end{equation*}
    Plugging these choices into the variational formulation and using \cref{eq:capupper} and \cref{thm:capsymm} gives:
    \begin{equation*}
        \capa_\varepsilon^\dag (\calA, \calB) \leq \inf_{f \in \rmH^1_{0}(\overline{\mathcal{B}}^\complement) ,\, f \geq \mathcal{X}_\calA} J_{\varepsilon, \overline{\calB}^\complement}[f] \leq \varepsilon \pair{ \widehat{\mathbf{S}} \widehat{\nabla} h_\varepsilon^\dag, \widehat{\mathbf{S}} \widehat{\nabla} h_\varepsilon^\dag }_\Omega \leq \capa_\varepsilon^\dag (\calA, \calB),
    \end{equation*}
    which completes the proof.
\end{proof}
}

\section{Rough Bounds on the Equilibrium Potential}\label{sec:roughestimate}
Throughout this section, we assume \cref{assump:elliptic}, namely $\mathbf{A}$ is uniformly elliptic and H\"older continuous, which automatically implies \cref{assump:equilibrium} from the classical regularity theory for second-order elliptic equations.

Following the pioneering work in \cite{BEGK04}, to prove the quantitative capacity, we first need to establish rough bounds on the equilibrium potential. This is a crucial first step. We will use the variational formula \cref{eq:variationalcap} to derive some rough bounds on the equilibrium potentials \( h_\varepsilon \) and \( h_\varepsilon^\dag \). Other approaches to establishing these rough bounds can be found in \cite{LMS19, LS22}. Our approach of utilizing the non-reversible variational formula is, to the best of our knowledge, novel.

The rough estimate relies on the following well-known renewal estimate, which has been treated in \cite{BEGK04,LMS19}.
\begin{proposition}\label{thm:renewal}
    Let \cref{assump:ab,assump:elliptic} hold, and let \( z \in \Omega \) be such that \( \mathrm{dist}(z, \overline{\calA \cup \calB}) > c\varepsilon \). Let \( B_r(z) \) denote the ball of radius \( r \) centered at \( z \). Then, for any \( r \leq c\varepsilon \), where \( c < \infty \), there exists a finite positive constant \( C \) (depending only on \( c \) and \( \norm{\nabla W(z)}_\infty \)) such that
    \begin{equation}\label{eq:renewal}
        h_{\calA, \calB}(z) \leq C \frac{\capa(B_r(z), \overline{\calA})}{\capa(B_r(z), \overline{\calB})}.
    \end{equation} 
\end{proposition}

To utilize \cref{eq:renewal}, we now prove some rough bounds on capacities.

The lower bound is easier to establish due to the following observation: from \cref{thm:variationalcap} and \cref{assump:elliptic}, we obtain, for a constant \( C \) that depends only on the ellipticity constant, the inequality
\begin{equation*}
    \capa_\varepsilon(\calA, \calB) = \varepsilon \int_{\overline{\calB}^\complement} \nabla h_\varepsilon \cdot \mathbf{S} \nabla h_\varepsilon \, \rho_\varepsilon(\rmd z) \geq C \varepsilon \int_{\overline{\calB}^\complement} \abs{\nabla h_\varepsilon}^2 \, \rho_\varepsilon(\rmd z) \geq C \capa^{\mathrm{symm}}_\varepsilon(\calA, \calB),
\end{equation*}
since by definition,
\begin{equation*}
    \capa^{\mathrm{symm}}_\varepsilon(\calA, \calB) = \inf_{f \in \rmH^1_0(\overline{\calB}^\complement), \, f \geq \mathcal{X}_\calA} \varepsilon \int_{\overline{\calB}^\complement} \abs{\nabla f}^2 \, \rho_\varepsilon(\rmd z).
\end{equation*}
Using the lower bound estimate from \cite{BEGK04}, we automatically obtain the following result:

\begin{lemma}\label{thm:roughlower}
    Let \cref{assump:ab,assump:elliptic} hold, and let \( z \in \overline{\calA}^\complement \). Denote by \( z^\star = z^\star(z, \overline{\calA}) \) a point such that
    \begin{equation*}
        W(z^\star) = \inf_{\gamma(t)} \sup_{t \in [0,1]} W(\gamma(t)),
    \end{equation*}
    where the infimum is taken over all continuous paths \( \gamma(t) \) leading from \( z \) to \( \overline{\calA} \).
    Then, there exists a constant \( C \) independent of \( \varepsilon \) such that
    \begin{equation*}
        \capa_\varepsilon(B_\varepsilon(z), \overline{\calA}) \geq C \varepsilon^d e^{-W(z^\star)/\varepsilon}.
    \end{equation*}
\end{lemma}

However, for the upper bound, we are unable to find a direct comparison between the reversible and non-reversible capacities when \( \bm{l} \neq 0 \). Nevertheless, we can prove the upper bound by constructing a rough candidate \( f \) and the corresponding admissible vector field \( \bm{g} \), and then plugging them into the variational formula \cref{eq:variationalcap}.

\begin{lemma}\label{thm:roughupper}
    Let \cref{assump:ab,assump:elliptic} hold, and let \( z \in \overline{\calA}^\complement \) be such that \( \mathrm{dist}(z, \calA) \geq c > \varepsilon \) for some fixed small constant \( c \) independent of \( \varepsilon \). Denote by \( z^\star = z^\star(z, \overline{\calA}) \) a point such that
    \begin{equation*}
        W(z^\star) = \inf_{\gamma(t)} \sup_{t \in [0, 1]} W(\gamma(t)),
    \end{equation*}
    where the infimum is taken over all continuous paths leading from \( z \) to \( \overline{\calA} \). Then, there exists a constant \( C \) independent of \( \varepsilon \) such that
    \begin{equation*}
        \capa_\varepsilon(B_\varepsilon(z), \overline{\calA}) \leq C \varepsilon^{d-4} e^{-W(z^\star)/\varepsilon}.
    \end{equation*}
\end{lemma}

\begin{remark}
    The condition \( \mathrm{dist}(z, \calA) \geq c \) is only a technical assumption to avoid intersection. It can be removed by choosing \( B_r(z) \) for arbitrarily small \( r \).
\end{remark}

\begin{proof}
    If \( z^\star = z \), we can construct a function \( f \) that is 1 on \( B_\varepsilon(z) \) and decays to 0 on \( B_{2\varepsilon}(z) \setminus B_\varepsilon(z) \), ensuring that \( \abs{\nabla f} \leq 1/\varepsilon \).

    On the other hand, if \( z^\star = 0 \), we construct \( f \) such that it is 1 on the whole connected component of the level set \( \{ y \in \mathbb{R}^d : W(y) < W(z) \} \) that contains \( z \), denoted by \( \mathcal{V} \), and changes from 1 to 0 on \( (\mathcal{V})_\varepsilon \setminus \mathcal{V} \), with \( \abs{\nabla f} \leq 1/\varepsilon \), where \( (\mathcal{V})_\varepsilon := \{ y \in \mathbb{R}^d : \mathrm{dist}(y, \mathcal{V}) < \varepsilon \} \). Note that for \( y \in \partial\mathcal{V} \), we have \( W(y) = H \).

    It remains to construct the admissible vector field \( \bm{g} \in \rmV_{\varepsilon, \overline{\calB}^\complement} [f] \). To do this, consider the following equation:
    \begin{equation*}
        \begin{cases}
            \varepsilon e^{W/\varepsilon} \nabla \cdot \left( e^{-W/\varepsilon} \nabla u \right) = \bm{l} \cdot \nabla f, & \text{in } \overline{\calA}^\complement, \\
            u = 0, & \text{on } \partial \calA.
        \end{cases}
    \end{equation*}
    It follows that there is a weak solution \( u \) with the following energy estimate:
    \begin{equation*}
        \varepsilon \int_{\overline{\calA}^\complement} \abs{\nabla u}^2 \, \rho_\varepsilon(\rmd z) = \int_{\overline{\calA}^\complement} u \bm{l} \cdot \nabla f \, \rho_\varepsilon(\rmd z) \leq \varepsilon \int_{\mathrm{supp}(\nabla f)} u^2 \, \rho_\varepsilon(\rmd z) + \frac{1}{\varepsilon} \int_{\overline{\calA}^\complement} \abs{\bm{l} \cdot \nabla f}^2 \, \rho_\varepsilon(\rmd z).
    \end{equation*}
    For both cases, denote 
    \begin{equation*}
        \underline{H} := \inf_{y \in \mathrm{supp}(\nabla f)} W(y), \quad \overline{H} := \sup_{y \in \mathrm{supp}(\nabla f)} W(y).
    \end{equation*}
    It follows that
    \begin{align*}
        \frac{1}{Z_\varepsilon} \int_{\mathrm{supp}(\nabla f)} \abs{u(y)}^2 e^{-W(y)/\varepsilon} \, \rmd y
        &\leq \frac{ \exp(-\underline{H}/\varepsilon) }{Z_\varepsilon} \int_{\{ W(y) < \overline{H} \} \setminus \calA} \abs{u(y)}^2 e^{-W(y)/\varepsilon} e^{W(y)/\varepsilon} \, \rmd y \\
        &\leq \frac{ \exp(-\underline{H}/\varepsilon) }{Z_\varepsilon} e^{\overline{H}/\varepsilon} C_P \int_{\overline{\calA}^\complement} \abs{\nabla u(y)}^2 e^{-W(y)/\varepsilon} \, \rmd y \\
        &\leq C \int_{\overline{\calA}^\complement} \abs{\nabla u}^2 \, \rho_\varepsilon(\rmd z),
    \end{align*}
    where, in the last inequality, we use the fact that \( \overline{H} - \underline{H} \leq c_1 \varepsilon \) for some constant \( c_1 \) depending only on the Lipschitz constant of \( W \). Moreover, \( C_P \) is the Poincaré constant, which depends only on the diameter of the domain \( \{ W(y) < \overline{H} \} \setminus \calA \), and \( C_P \approx 1/\mathrm{dist}(z, \overline{\calA}) \). 
    Indeed, the Poincaré inequality is obtained by connecting points in the domain to points on the boundary via paths and applying fundamental theorem of calculus. 
    The Poincare constant depends on the length of paths needed to connect all points in the domain. 
    Hence, as long as \( \mathrm{dist}(z, \calA) \geq c \) for some constant \( c \) independent of \( \varepsilon \), \( C_P \leq 1/c \).

    Finally, by defining \( \bm{g} = \nabla u \), we have
    \begin{equation*}
        \varepsilon \int_{\overline{\calA}^\complement} \abs{\bm{g}}^2 \, \rho_\varepsilon(\rmd z) \leq \frac{C}{\varepsilon} \int_{\overline{\calA}^\complement} \abs{\bm{l} \cdot \nabla f}^2 \, \rho_\varepsilon(\rmd z) \leq  \frac{C}{\varepsilon} \varepsilon^{d-1}\abs{\partial\calA} \frac{1}{\varepsilon^2} \sup_{y\in\mathrm{supp}(f)}e^{-W(y)}\leq  C \varepsilon^{d-4} e^{-W(z^\star)/\varepsilon},
    \end{equation*}
    where the constant \( C \) is independent of \( \varepsilon \).

    Recall the variational formula \cref{eq:variationalcap}. We then have for a constant \( C \) independent of \( \varepsilon \) that
    \begin{multline*}
        \capa_\varepsilon (B_\varepsilon(z), \overline{\calA}) \leq \varepsilon \left\langle \mathbf{A}\nabla f - \bm{g}, \mathbf{A} \nabla f - \bm{g} \right\rangle_{\overline{\calB}^\complement} \\
        \leq C \varepsilon \int_{\overline{\calB}^\complement} \abs{\nabla f}^2 \, \rho_\varepsilon(\rmd z) + \varepsilon \int_{\overline{\calB}^\complement} \abs{\bm{g}}^2 \, \rho_\varepsilon(\rmd z)
        \leq C \varepsilon^{d-4} e^{-W(z^\star)/\varepsilon},
    \end{multline*}
    which completes the proof.
\end{proof}

As a corollary, we easily obtain rough bounds on the equilibrium potentials.

\begin{corollary}\label{thm:roughbound}
    Let \cref{assump:ab,assump:elliptic,assump:twowell} hold, and let \( \mathcal{V}_1 \) be the valley containing \( \bm{m}_1 \) and \( \mathcal{V}_2 \) be the valley containing \( \bm{m}_2 \). Recall that \( h_\varepsilon \) and \( h_\varepsilon^\dag \) are solutions to the equilibrium problem \cref{eq:equilibriumproblem}. There exists a constant \( C \) independent of \( \varepsilon \) such that:
    \begin{itemize}
        \item If \( z \in \mathcal{V}_0 \), then
        \[
            h_\varepsilon(z), h_\varepsilon^\dag(z) \leq C \varepsilon^{-4} \exp\left( -\frac{H - U(z)}{\varepsilon} \right);
        \]
        \item If \( z \in \mathcal{V}_1 \), then
        \[
            h_\varepsilon(z), h_\varepsilon^\dag(z) \geq 1 - C \varepsilon^{-4} \exp\left( -\frac{H - U(z)}{\varepsilon} \right).
        \]
    \end{itemize}
\end{corollary}

\begin{proof}
    Applying \cref{thm:renewal,thm:roughlower,thm:roughupper} completes the proof of the first estimate. For the second estimate, we use the fact that \( h_{\varepsilon; \calB, \calA} = 1 - h_{\varepsilon; \calA, \calB} \).
\end{proof}

\section{Sharp Capacity Estimates}\label{sec:sharpestimate}
Throughout this section, we assume \cref{assump:elliptic}, namely $\mathbf{A}$ is uniformly elliptic and H\"older continuous, which automatically implies \cref{assump:equilibrium} from the classical regularity theory for second-order elliptic equations.

To prove \cref{thm:capestimate}, the general approach is to find good approximations of the equilibrium potentials and use \cref{thm:variationalcap} to establish upper and lower bounds. However, as shown in \cite{LMS19}, using the full variational formulation is somewhat redundant, since the key estimates is the same as in \cite{LS22}, which, in turn, rely on the observation \cref{eq:capineq}. In the following, we closely follow \cite{LS22} to compute the capacity, with the only minor difference being that we will perform direct calculations on the mollified functions, rather than constructing step-by-step approximations. We will list only the essential modifications and refer to \cite{LS22} for further details.

The core idea is to construct an approximation of the equilibrium potential \( h_\varepsilon \) through linearization around the saddle point. This is because the equilibrium is expected to be 1 in \( \calA = B_\varepsilon(\bm{m}_1) \) and 0 in \( \calB = B_\varepsilon(\bm{m}_0) \), and the transition region is concentrated around the saddle point due to the weight \( \exp(-W/\varepsilon) \).

\subsection{Linearization Near the Saddle Point and the Geometry}
In a small neighborhood of the origin, the linearized equations are given by
\begin{align*}
    \widetilde{\L}_\varepsilon f &:= \varepsilon \nabla \cdot (\mathbf{A}_0 \nabla f) - \nabla f \cdot \left( \mathbf{A}_0^\transpose \mathbf{H}_0 + \mathbf{L}_0 \right) z, \\
    \widetilde{\L}_\varepsilon^\dag f &:= \varepsilon \nabla \cdot (\mathbf{A}_0^\transpose \nabla f) - \nabla f \cdot \left( \mathbf{A}_0 \mathbf{H}_0 - \mathbf{L}_0 \right) z,
\end{align*}
where \( \mathbf{A}_0 := \mathbf{A}(0) \), and recall that \( \mathbf{H}_z = \nabla^2 W(z) \) and \( \mathbf{L}_z = \nabla \bm{l}(z) \).

By \cref{assump:twowell}, \( \mathbf{H}_0 \) has only one negative eigenvalue. Let \( -\lambda_1, \lambda_2, \dots, \lambda_N \) denote the eigenvalues of \( \mathbf{H}_0 \), and let \( \mathbf{e}_k \) be the canonical basis vectors corresponding to the eigenvalues \( \lambda_k \). Without loss of generality, assume that \( \bm{e}_1 \) is directed toward the valley containing \( \bm{m}_1 \).

We first introduce, for a constant \( K \) to be chosen later, the following quantity:
\[
    \delta := K \sqrt{\varepsilon \log(1/\varepsilon)}.
\]

We split the domain into the following regions, similar to the approach in \cite{LMS19,LS22}.
Let \( \mathcal{Q} \) be a closed hyperrectangle around the saddle point 0, defined by
\[
    \mathcal{Q} = [-\delta, \delta] \times \prod_{i=2}^d \left[-\sqrt{2\lambda_1 / \lambda_i} \delta, \sqrt{2\lambda_1 / \lambda_i} \delta \right].
\]
We write the boundaries of \( \mathcal{Q} \) as
\[
    \partial_+ \mathcal{Q} = \{ z \in \mathcal{Q} : z_1 = \delta \}, \quad \partial_- \mathcal{Q} = \{ z \in \mathcal{Q} : z_1 = -\delta \}, \quad \text{and} \quad \partial_\perp \mathcal{Q} = \partial \mathcal{Q} \setminus (\partial_+ \mathcal{Q} \cup \partial_- \mathcal{Q}).
\]
Recall that \( H = W(0) \) is the saddle height, and define the super level set $\mathcal{O}$ and the bridge set around the saddle, denoted by $\mathcal{S}$, in the following way: 
\[
    \mathcal{O} := \{ U > H + \lambda_1 \delta^2 / 4 \}, \quad \text{and} \quad \mathcal{S} := \mathcal{Q} \cap \mathcal{O}^\complement.
\]
The right and left boundaries of \( \mathcal{S} \) are denoted by \( \partial_+ \mathcal{S} = \partial \mathcal{S} \cap \partial_+ \mathcal{Q} \) and \( \partial_- \mathcal{S} = \partial \mathcal{S} \cap \partial_- \mathcal{Q} \), respectively.

It is clear that the saddle region \( \mathcal{S} \) splits the domain \( \mathcal{O}^\complement \) into two valleys, \( \mathcal{V}_1 \) and \( \mathcal{V}_0 \), which contain \( \bm{m}_1 \) and \( \bm{m}_0 \), respectively.

\begin{lemma}\label{thm:hml}
    The matrix \( \mathbf{H}_0\mathbf{A}_0+\mathbf{L}_0^\transpose \) has a unique negative eigenvalue \( -\mu \) with a unit eigenvector \( \bm{v} \). 

    Similarly, \( \mathbf{H}_0\mathbf{A}_0^\transpose - \mathbf{L}_0^\transpose \) has the same unique negative eigenvalue \( -\mu \) with a unit eigenvector \( \bm{v}^\dag \). 
\end{lemma}
{
\begin{proof}
    First, from \cite[Lemma 4.5]{LS22}, \( \mathbf{H}_0\mathbf{L}_0 \) is skew-symmetric, i.e.~$
    (\mathbf{H}_0\mathbf{L}_0)^\transpose = -\mathbf{H}_0\mathbf{L}_0.
    $

    Secondly, \( \mathbf{H}_0\mathbf{A}_0^\transpose -\mathbf{L}_0^\transpose \) is similar to \( \mathbf{A}_0^\transpose\mathbf{H}_0+\mathbf{L}_0 \). Indeed, 
    \begin{equation*}
        \mathbf{H}_0^{-1}\left(\mathbf{H}_0\mathbf{A}_0^\transpose - \mathbf{L}_0^\transpose\right)\mathbf{H}_0 
        = \mathbf{H}_0^{-1} \left(\mathbf{H}_0\mathbf{A}_0^\transpose\mathbf{H}_0 + \mathbf{H}_0\mathbf{L}_0\right) 
        = \mathbf{A}_0^\transpose\mathbf{H}_0+\mathbf{L}_0.
    \end{equation*}

    Next, we substitute \( \mathbf{A}_0^\transpose + \mathbf{L}_0\mathbf{H}_0^{-1} \) for \( \mathbf{A} \) and \( \mathbf{H}_0 \) for \( \mathbf{B} \) in \cite[Lemma 4.2]{LS22} to conclude that \( \mathbf{A}_0^\transpose\mathbf{H}_0 + \mathbf{L}_0 \) has a unique negative eigenvalue \( -\mu \). 
    Consequently, \( \mathbf{H}_0\mathbf{A}_0^\transpose - \mathbf{L}_0^\transpose \) has the same unique negative eigenvalue \( -\mu \).

    Finally, it is easy to see that 
    $
    (\mathbf{H}_0\mathbf{A}_0 + \mathbf{L}_0^\transpose)^\transpose = \mathbf{A}_0^\transpose\mathbf{H}_0 + \mathbf{L}_0,
    $
    which completes the proof.
\end{proof}
}

Next, we define
\begin{equation}\label{eq:beta}
    \beta = \frac{\mu}{\bm{v}\cdot\mathbf{A}_0\bm{v}},\quad 
    \beta^\dag = \frac{\mu}{\bm{v}^\dag\cdot\mathbf{A}_0\bm{v}^\dag}.
\end{equation}
The construction of these constants comes from \cite{LS18,LMS19}.

\begin{lemma}\label{thm:evector}
    We have
    \begin{equation*}
        \bm{v}\cdot\mathbf{H}^{-1}_{0}\bm{v} = -\frac{v_1^2}{\lambda_1} + \sum_{k=2}^d\frac{v_k^2}{\lambda_k} = -\frac{1}{\beta} < 0.
    \end{equation*}
    Hence, \( (\bm{v}\cdot\bm{e}_1)^2 > 0 \), meaning they cannot be orthogonal. Similarly, \( (\bm{v}^\dag\cdot\bm{e}_1)^2 > 0 \).
\end{lemma}
\begin{proof}
    Combine \cite[Lemma 8.1]{LS22} and \cite[Lemma 8.1]{LMS19}. 
\end{proof}

Next, we define a linear function around the saddle point \( z=0 \) as follows:
\begin{equation}\label{eq:peps}
    p_\varepsilon(z) = \frac{1}{C_\varepsilon} \int_{-\infty}^{z\cdot\bm{v}} \exp\left(-\frac{\beta}{2\varepsilon}t^2\right)\,\rmd t,\quad z\in\mathcal{S},
\end{equation}
where
\begin{equation*}
    C_\varepsilon := \int_{-\infty}^\infty \exp\left(-\frac{\beta}{2\varepsilon}t^2\right)\,\rmd t = \sqrt{\frac{2\pi \varepsilon}{\beta}}.
\end{equation*}
Similarly, we define 
\begin{equation}\label{eq:pepsdag}
    p_\varepsilon^\dag(z) = \frac{1}{C_\varepsilon^\dag} \int_{-\infty}^{z\cdot\bm{v}^\dag} \exp\left(-\frac{\beta^\dag}{2\varepsilon}t^2\right)\,\rmd t,\quad z\in\mathcal{S},
\end{equation}
where
\begin{equation*}
    C_\varepsilon^\dag := \int_{-\infty}^\infty \exp\left(-\frac{\beta^\dag}{2\varepsilon}t^2\right)\,\rmd t = \sqrt{\frac{2\pi \varepsilon}{\beta^\dag}}.
\end{equation*}

Finally, we extend \( p_\varepsilon \) and \( p_\varepsilon^\dag \) to be 1 on \( \mathcal{V}_1 \) and 0 on \( \mathcal{V}_0 \cup \mathcal{O} \). 
This serves as a good approximation to the equilibrium potentials.

\subsection{Proof of Theorem 3}
We first mollify the constructed functions \( p_\varepsilon \) and \( p_\varepsilon^\dag \) to obtain \( p_{\varepsilon,\eta} \) and \( p_{\varepsilon,\eta}^\dag \), respectively, in order to remove jumps. Specifically, we define  
\begin{equation*}
    p_{\varepsilon,\eta} = p_\varepsilon * \varphi_\eta, \quad p_{\varepsilon,\eta}^\dag = p_\varepsilon^\dag * \varphi_\eta,  
\end{equation*}
where \( \varphi_\eta \) is the standard mollifier such that $\mathrm{supp}(\varphi_\eta)\subset B_\eta(0)$.

Given the variational formulation \cref{eq:variationalcap}, a natural approach is to find a good approximation of the minimizer. However, constructing a suitable admissible vector field associated with \( p_{\varepsilon,\eta} \) is highly challenging. Instead, we rely on the following key observation: for any \( f \in \rmH^1_0(\overline{\calB}^\complement) \) with \( f \geq \mathcal{X}_{\calA} \), the vector field \( \bm{l} f / \varepsilon \) is a trivial candidate for the admissible vector field, since by \cref{assump:divfree}, we have
\begin{equation*}
    \varepsilon e^{W/\varepsilon} \nabla\cdot\left(e^{-W/\varepsilon} {\bm{l}f}/{\varepsilon}\right) = \bm{l}\cdot\nabla f.
\end{equation*}
It follows from \cref{eq:capineq} that for any \( f \in \rmH^1_0(\overline{\calB}^\complement) \) such that \( f = 1 \) on \( \partial\calA \), we have  
\begin{equation}
    \varepsilon\pair*{ \mathbf{A} \nabla f - \bm{l} f/\varepsilon, \mathbf{S} \nabla h_\varepsilon^\dag }_{\overline{\calB}^\complement} = \capa_\varepsilon(\calA,\calB).
\end{equation}

\begin{remark}
    This key observation originally appeared in \cite{LMS19,LS22}. However, the vector field \( \bm{l} f / \varepsilon \) is not well-suited for the variational formula \cref{eq:variationalcap} since it is not close to the minimizer. 
\end{remark}

\begin{delayedproof}{thm:capestimate}
Now, consider the constructed approximation of the equilibrium potential, \( p_{\varepsilon,\eta} \) and \( p_{\varepsilon,\eta}^\dag \). By a straightforward integration by parts, we obtain  
\begin{align*}
    \capa_\varepsilon(\calA,\calB) 
    =& \varepsilon\pair*{ \mathbf{A}_{0} \nabla p_{\varepsilon,\eta} - \bm{l} p_{\varepsilon,\eta}/\varepsilon, \mathbf{S}\nabla h_\varepsilon^\dag }_{\overline{\calB}^\complement} 
    + \varepsilon\pair*{ (\mathbf{A}(z)-\mathbf{A}_{0}) \nabla p_{\varepsilon,\eta} , \mathbf{S}\nabla h_\varepsilon^\dag }_{\overline{\calB}^\complement} 
    \\
    =& \underbrace{-\int_{\Omega} h_\varepsilon^\dag \L_{\varepsilon,0} p_{\varepsilon,\eta}\,\rho_\varepsilon(\rmd z)}_{\mathrm{I}} + \underbrace{\varepsilon\int_{\Omega} (\mathbf{A}(z)-\mathbf{A}_{0})\nabla p_{\varepsilon,\eta} \cdot\nabla h_\varepsilon^\dag \,\rho_\varepsilon(\rmd z)}_{\mathrm{II}},
\end{align*}
where  
\begin{equation*}
    \L_{\varepsilon,0} f = \varepsilon e^{U/\varepsilon} \nabla\cdot\left[ e^{-U/\varepsilon} \left(\mathbf{A}_{0} \nabla f - \frac{1}{\varepsilon} \bm{l}f\right)\right]
\end{equation*}
is the freezing coefficient operator. 

Next, we define the geometric notation: for any set \( \mathcal{S} \), we define the $\eta$-enlargement by 
\begin{equation*}
    (\mathcal{S})_\eta = \left\{z\in\R^d: \mathrm{dist}(z, \mathcal{S}) < \eta\right\}.
\end{equation*}
If \( \mathcal{S} \) has a non-empty interior, we define its $\eta$-shrinkage by
\begin{equation*}
    (\mathcal{S})_{-\eta} := \R^d\setminus\left(\mathcal{S}^\complement\right)_{\eta}.
\end{equation*}

Using this geometric decomposition, we split the integrals accordingly:  
\begin{equation*}
    \mathrm{I} = \mathrm{I}_1 + \mathrm{I}_2 + \mathrm{I}_3 + \mathrm{I}_4,
\end{equation*}
where  
\begin{align*}
    &\mathrm{I}_1 = -\int_{(\partial_+\mathcal{S})_{\eta}} h_\varepsilon^\dag \L_{\varepsilon,0} p_{\varepsilon,\eta} \,\rho_\varepsilon(\rmd z),
    &\mathrm{I}_2 = -\int_{(\partial_-\mathcal{S})_{\eta}} h_\varepsilon^\dag \L_{\varepsilon,0} p_{\varepsilon,\eta} \,\rho_\varepsilon(\rmd z),\\
    &\mathrm{I}_3 = -\int_{(\mathcal{S})_{-\eta}} h_\varepsilon^\dag \L_{\varepsilon,0} p_{\varepsilon,\eta} \,\rho_\varepsilon(\rmd z),
    &\mathrm{I}_4 = -\int_{(\partial\mathcal{O})_{\eta}\setminus(\partial_\pm\mathcal{S})_\eta} h_\varepsilon^\dag \L_{\varepsilon,0} p_{\varepsilon,\eta} \,\rho_\varepsilon(\rmd z),
\end{align*}
and 
\begin{equation*}
    \abs{\mathrm{II}} \leq \mathrm{II}_1 + \mathrm{II}_2 + \mathrm{II}_3 + \mathrm{II}_4 + \mathrm{II}_5,
\end{equation*}
where 
\begin{align*}
    &\mathrm{II}_1 = \delta^\alpha \varepsilon \int_{(\mathcal{S})_{-\eta}} \abs{\nabla p_{\varepsilon,\eta}}^2\,\rho_\varepsilon(\rmd z),
    &\mathrm{II}_2 = \delta^\alpha \varepsilon \int_{(\partial_+\mathcal{S})_{\eta}\cup(\partial_-\mathcal{S})_{\eta}} \abs{\nabla p_{\varepsilon,\eta}}^2\,\rho_\varepsilon(\rmd z),
    \\
    &\mathrm{II}_3 = \delta^\alpha \varepsilon \int_{(\mathcal{S})_\eta} \abs{\nabla h_\varepsilon^\dag}^2\,\rho_\varepsilon(\rmd z),
    &\mathrm{II}_4 = c_1 \varepsilon \int_{(\partial\mathcal{O})_\eta} \abs{\nabla p_{\varepsilon,\eta}}^2 \,\rho_\varepsilon(\rmd z),
    \\
    &\mathrm{II}_5 = c_1 \varepsilon \int_{(\partial \mathcal{O})_\eta} \abs{ \nabla h_\varepsilon^\dag }^{2} \,\rho_\varepsilon(\rmd z), &
\end{align*}
and \( c_1 \) is defined as
\begin{equation*}
    c_1:= \sup_{x\in(\partial\mathcal{O})_\eta} \norm*{\mathbf{A}(z) - \mathbf{A}(0)}_\infty.
\end{equation*}

We first claim that the main contribution comes from \( \mathrm{I}_1 \), while all remaining terms are small in comparison. To simplify notation, we introduce  
\begin{equation*}
    A_\varepsilon = \frac{(2\pi\varepsilon)^{d/2}}{Z_\varepsilon}e^{-H/\varepsilon}, \quad \omega_0 = \frac{\mu}{2\pi\sqrt{-\det(\mathbf{H}_0)}}.
\end{equation*}

\begin{proposition}\label{thm:estimates}
    For $K$ sufficiently large (chosen in the proof and independent of $\varepsilon$),
    we establish the following estimates:
    \begin{align*}
        &\mathrm{I}_1 = (1+o_\varepsilon(1)) A_\varepsilon \omega_0,\\
        &\mathrm{II}_{3} = \delta^\alpha \capa_\varepsilon(\calA,\calB) = o_\varepsilon(1)\capa_\varepsilon(\calA,\calB),\\
        &\mathrm{II}_{5} = o_\varepsilon(1) \sqrt{A_\varepsilon \capa_\varepsilon(\calA,\calB)} + o_\varepsilon(1) A_\varepsilon \omega_0,\\
        &\text{All other terms} = o_\varepsilon(1) A_\varepsilon \omega_0. 
    \end{align*}
\end{proposition}
The proof of these estimates is postponed to \cref{sec:estimate1,sec:estimate2}.

Finally, combining the above estimates, we obtain  
\begin{equation*}
    (1+o_\varepsilon(1))\capa_\varepsilon = (1+o_\varepsilon(1))A_\varepsilon \omega_0 + o_\varepsilon(1) \sqrt{A_\varepsilon \capa_\varepsilon}.
\end{equation*}
Solving this equation (see e.g.~\cite{LS22}), we conclude  
\begin{equation*}
    \capa_\varepsilon = (1 + o_\varepsilon(1)) A_\varepsilon \omega_0,
\end{equation*}
which proves \cref{thm:capestimate}.
\end{delayedproof}

\subsection{Estimation of I}\label{sec:estimate1}
We provide a detailed estimate for \(\mathrm{I}_1\). 
First, we expand the term using integration by parts on convolutions and use linearization, splitting it further into five terms:
\begin{equation*}
    \mathrm{I}_1 = -\int_{(\partial_+\mathcal{S})_{\eta}} h_\varepsilon^\dag \L_{\varepsilon,0} p_{\varepsilon,\eta} \,\rho_\varepsilon(\rmd z) = 
    \mathrm{I}_{11} + \cdots + \mathrm{I}_{15},
\end{equation*}
where
\begin{align*}
    &\mathrm{I}_{11} := \int_{(\partial_+\mathcal{S})_\eta} h_\varepsilon^\dag \left(\int_{B_\eta(z)\cap \partial\mathcal{S}} \left(\varepsilon \mathbf{A}_0 \nabla p_\varepsilon(y) + (p_\varepsilon(y)- 1) \mathbf{L}_0 y\right)\cdot\bm{n}_{\mathcal{S}} \varphi_\eta(x-y) \,\sigma(\rmd y) \right)\rho_\varepsilon(\rmd z),
    \\
    &\mathrm{I}_{12} := -\int_{(\partial_+\mathcal{S})_\eta} h_\varepsilon^\dag \left(\int_{B_\eta(z)\cap \mathcal{S}} \widetilde{\L}_{\varepsilon} p_\varepsilon(y) \varphi_\eta(x-y) \,\rmd y \right)\rho_\varepsilon(\rmd z),
    \\
    &\mathrm{I}_{13} := \int_{(\partial_+\mathcal{S})_\eta} h_\varepsilon^\dag \left(\int_{B_\eta(z)\cap \partial\mathcal{S}} \varepsilon (p_\varepsilon(y) - 1) \mathbf{A}_0 \bm{n}_{\mathcal{S}}\cdot \nabla_y \varphi_\eta(x-y) \,\sigma(\rmd y) \right)\rho_\varepsilon(\rmd z),
    \\
    &\mathrm{I}_{14} := -O(\delta^2 + \eta) \int_{(\partial_+\mathcal{S})_\eta} h_\varepsilon^\dag \left(\int_{B_\eta(z)\cap \mathcal{S}} \abs{\nabla p_\varepsilon(y)} \varphi_\eta(x-y) \,\rmd y \right)\rho_\varepsilon(\rmd z),
    \\
    &\mathrm{I}_{15} := \int_{(\partial_+\mathcal{S})_\eta} h_\varepsilon^\dag \mathbf{A}_0^\transpose \nabla W(z)  \left(\int_{B_\eta(z)\cap \partial\mathcal{S}} (p_\varepsilon(y) - 1) \varphi_\eta(x-y) \bm{n}_{\mathcal{S}} \,\sigma(\rmd y) \right)\rho_\varepsilon(\rmd z).
\end{align*}

Secondly, by the construction of \(p_\varepsilon\) in \cref{eq:peps}, we obtain
\begin{equation}\label{eq:I12}
    \mathrm{I}_{12} = 0.
\end{equation}

Next, for the term \(\mathrm{I}_{14}\), using \(0\leq h_\varepsilon^\dag \leq 1\) and following the proof of \cite[Lemma 8.7]{LMS19}, we derive
\begin{equation}\label{eq:I14}
    \abs{\mathrm{I}_{14}} \leq O(\delta^2) \int_{\mathcal{S}} \abs{\nabla p_\varepsilon} \rho_\varepsilon(\rmd z) = o_\varepsilon(1) A_\varepsilon.
\end{equation}

For \(\mathrm{I}_{13}\) and \(\mathrm{I}_{15}\), using \(0\leq h_\varepsilon^\dag \leq 1\), the boundedness of coefficients, and the result from \cite[Assertion 8.A]{LMS19}, which states that for \(y\in\partial_+\mathcal{S}\), we have 
\begin{equation*}
    \abs{p_\varepsilon(y)-1}  \leq \varepsilon^{c_2 K^2 } e^{-H/\varepsilon}
\end{equation*}
for some positive constant \(c_2\), where \(K\) can be chosen sufficiently large, we conclude that
\begin{equation}\label{eq:I1315}
    \mathrm{I}_{13}, \mathrm{I}_{15} = o_\varepsilon(1) A_\varepsilon.
\end{equation}

The dominant contribution comes from the surface integral \(\mathrm{I}_{11}\).
Using the estimate \(1 - h_\varepsilon^\dag(y) \leq o_\varepsilon(1)\) for \(y\in\partial_+\mathcal{S}\) from \cref{thm:roughbound}, the construction of \(p_\varepsilon\) in \cref{eq:peps}, the matrix structure and eigenvalues from \cref{thm:hml,thm:evector,eq:beta}, and following the detailed steps in \cite[Prop.~8.6]{LS22}, we obtain 
\begin{equation}\label{eq:I11}
    \mathrm{I}_{11} 
    = (1+o_\varepsilon(1)) \frac{1}{Z_\varepsilon} \frac{(2\pi\varepsilon)^{d/2}}{2\pi}\frac{\mu}{\sqrt{-\det (\mathbf{H}_{0})}}e^{-H/\varepsilon} = (1+o_\varepsilon(1)) A_\varepsilon \omega_0.
\end{equation}
Combining the estimates \cref{eq:I12,eq:I14,eq:I1315,eq:I11}, we conclude
\begin{equation*}
    \mathrm{I}_{1} = (1+o_\varepsilon(1)) A_\varepsilon \omega_0.
\end{equation*}

The fact that remaining terms \(\mathrm{I}_2,\mathrm{I}_3,\mathrm{I}_4\) are small compared to $\mathrm{I}_1$ follows from similar calculations as those for \(\mathrm{I}_1\). 
Specifically:
\begin{itemize}
    \item $\mathrm{I}_2$ is small because $h_\varepsilon^\dag$ is small by \cref{thm:roughbound};
    \item $\mathrm{I}_3$ is small because $\int_{\mathcal{S}}\widetilde{\L}_{\varepsilon} p_{\varepsilon}\,\rho_\varepsilon(\rmd z)$ is small, according to \cite[Lemma 8.7]{LMS19};
    \item $\mathrm{I}_4$ is small because the weight $\exp(-W/\varepsilon)$ is sufficiently small on $\partial\mathcal{O}$.
\end{itemize}
Combining these estimates yields the final result:
\begin{equation*}
    \mathrm{I}_2,\mathrm{I}_3,\mathrm{I}_4 = o_\varepsilon(1)A_\varepsilon \omega_0.
\end{equation*}
This completes the estimation of \(\mathrm{I}\).

\subsection{Estimation of II}\label{sec:estimate2}
First, from \cite[Lemma 8.4]{LMS19}, we obtain
\begin{equation*}
    \varepsilon \int_{\mathcal{S}} \nabla p_\varepsilon\cdot \mathbf{A}_0 \nabla p_\varepsilon \,\rho_\varepsilon(\rmd z) = (1+o_\varepsilon(1)) A_\varepsilon \omega_0.
\end{equation*}
Combining this with the calculations for mollification in \cref{sec:estimate1}, we deduce
\begin{equation}\label{eq:II1}
    \mathrm{II}_1 = \delta^\alpha (1+o_\varepsilon(1)) A_\varepsilon \omega_0.
\end{equation}
Next, using the smallness of \(\nabla p_{\varepsilon}\) on \(\partial\mathcal{S}_{\pm}\) and the smallness of \(\exp(-W(z)/\varepsilon)\) on \(\partial\mathcal{O}\), we obtain
\begin{equation}\label{eq:II24}
    \mathrm{II}_2, \mathrm{II}_4 = o_\varepsilon(1)A_\varepsilon \omega_0.
\end{equation}
Furthermore, from \cref{thm:variationalcap} and \cref{assump:elliptic}, it follows that
\begin{equation}\label{eq:II3}
    \mathrm{II}_3 \leq \delta^\alpha \capa_\varepsilon(\calA,\calB).
\end{equation}

It remains to estimate \(\mathrm{II}_5\). The key idea is that near \(\partial\mathcal{O}\), the weight \(\exp(-U(z)/\varepsilon)\) is small, which allows us to control the energy of \(\nabla h_\varepsilon^\dag\). To achieve this, we apply a standard energy estimate technique from regularity theory.

Since \(h_\varepsilon^\dag\) solves \cref{eq:equilibriumproblem}, it satisfies the weak formulation that for any test function \(\psi \in C_0^\infty(\Omega)\):
\begin{equation*}
    \int_{\Omega} \left(\varepsilon\mathbf{A}\nabla h_\varepsilon^\dag \cdot\nabla \psi - \psi \bm{l}\nabla h_\varepsilon^\dag  \right)\rho_\varepsilon(\rmd z) = 0,
\end{equation*} 
where \(\Omega = (\overline{\calA\cup\calB})^\complement\).
Let \(\phi\) be a cutoff function such that \(\phi = 1\) on \((\partial\mathcal{O})_{\eta}\) and \(\phi = 0\) outside \((\partial\mathcal{O})_{2\eta}\), with \(\abs{\nabla \phi}\leq c/\eta\) on \((\partial\mathcal{O})_{2\eta} \setminus (\partial\mathcal{O})_\eta\) for some constant \(c\).
Setting \(\psi = h_\varepsilon^\dag \phi\) in the weak formulation and expanding, we get
\begin{align*}
    0 &= \int_{\Omega} \left(\varepsilon\mathbf{A}\nabla h_\varepsilon^\dag \cdot\nabla (h_\varepsilon^\dag\phi) - (h_\varepsilon^\dag\phi) \bm{l}\nabla h_\varepsilon^\dag  \right)\rho_\varepsilon(\rmd z) \\
    &= \int_{\Omega} \varepsilon(\mathbf{A}\nabla h_\varepsilon^\dag \cdot\nabla h_\varepsilon^\dag)\phi \,\rho_\varepsilon(\rmd z) + \int_{\Omega} \varepsilon(\mathbf{A}\nabla h_\varepsilon^\dag \cdot\nabla \phi) h_\varepsilon^\dag \,\rho_\varepsilon(\rmd z) 
    - \int_{\Omega} \phi \bm{l}\nabla ((h_\varepsilon^\dag)^2 /2)  e^{U/\varepsilon}\,\rho_\varepsilon(\rmd z).
\end{align*}

For the second term, using $0\leq h_\varepsilon^\dag\leq 1$ and applying Hölder's inequality yields for some constant $c_2,c_3>0$ that
\begin{align*}
    \abs*{\int_{\Omega} \varepsilon (\mathbf{A}\nabla h_\varepsilon^\dag \cdot\nabla \phi)h_\varepsilon^\dag\,\rho_\varepsilon(\rmd z)} 
    &\leq c_2 \varepsilon \left(\int_{(\partial\mathcal{O})_{2\eta}} \abs{\nabla h_\varepsilon^\dag}\,\rho_\varepsilon(\rmd z)\right)^{1/2} 
    \left(\int_{(\partial\mathcal{O})_{2\eta}}\abs{\nabla \phi}^2\,\rho_\varepsilon(\rmd z)\right)^{1/2} \\
    &\leq c_2\sqrt{\capa_\varepsilon(\calA,\calB)} \left(\varepsilon^{c_3 K^2} e^{-H/\varepsilon}\right)^{1/2}.
\end{align*}

For the last term, integrating by parts (with no boundary contribution since \(\phi\) is compactly supported), we obtain
\begin{equation*}
    \abs*{\int_{\Omega} \phi \bm{l} \nabla((h_\varepsilon^\dag)^2/2) \,\rho_\varepsilon(\rmd z)} 
    \leq \int_{(\partial\mathcal{O})_{2\eta}} \abs{\bm{l}\cdot\nabla \phi}/2\,\rho_\varepsilon(\rmd z) 
    \leq c_2 \varepsilon^{c_3 K^2} e^{-H/\varepsilon}. 
\end{equation*}

Combining the above two estimates, we conclude
\begin{align*}
    \varepsilon\int_{(\partial\mathcal{O})_\eta} \abs{\nabla h_\varepsilon^\dag}\,\rho_\varepsilon(\rmd z) 
    &\leq \int_{(\partial\mathcal{O})_{\eta}} \varepsilon \mathbf{A} \nabla h_\varepsilon^\dag \cdot\nabla h_\varepsilon^\dag \,\rho_\varepsilon(\rmd z) 
    \\
    &\leq c_2 \sqrt{\capa_\varepsilon(\calA,\calB)} \left(\varepsilon^{c_3 K^2} e^{-H/\varepsilon}\right)^{1/2} + c_2 \varepsilon^{c_3 K^2} e^{-H/\varepsilon}.
\end{align*}
Choosing \(K\) sufficiently large, we obtain
\begin{equation}\label{eq:II5}
    \mathrm{II}_5 = o_\varepsilon(1)\sqrt{A_\varepsilon \capa_\varepsilon(\calA,\calB)} + o_\varepsilon(1) A_\varepsilon \omega_0.
\end{equation}

Collecting all the estimates \cref{eq:I11,eq:I12,eq:I1315,eq:I14,eq:II1,eq:II24,eq:II3,eq:II5} for \(\mathrm{I}\) and \(\mathrm{II}\) completes the proof of \cref{thm:estimates}.

\tocless{\section*{Acknowledgement}}
M.~Hou was supported by [Swedish Research Council dnr: 2019-04098].
The author wishes to thank Benny Avelin for fruitful discussions and revising the paper.

\printbibliography

\end{document}